\documentclass[12pt]{amsart}
\usepackage{frenchineq}
\usepackage{amsmath}
\usepackage{amsthm}
\usepackage{amssymb}
\usepackage{verbatim}
\usepackage{subfigure}
\usepackage{url}
\numberwithin{equation}{section}
\usepackage{textcomp}
\usepackage{stmaryrd}
\usepackage{xy}
\usepackage{color}

\usepackage{mathrsfs}
\usepackage{hyperref}

\usepackage{color}
\usepackage[colorinlistoftodos,bordercolor=orange,backgroundcolor=orange!20,linecolor=orange,textsize=scriptsize]{todonotes}
\usepackage{enumitem}
\usepackage{xspace}

\newenvironment{rcases}
  {\left.\begin{aligned}}
  {\end{aligned}\right\rbrace}

\definecolor{Mygrey}{gray}{0.75}

\def\displayandname#1{\rlap{$\displaystyle\csname #1\endcsname$}%
                      \qquad \texttt{\char92 #1}}

\makeatletter
\def\url@leostyle{
  \@ifundefined{selectfont}{\def\UrlFont{\sf}}{\def\UrlFont{\small\ttfamily}}}
\makeatother

\newcommand{\myemph}[1]{\emph{#1}}

\DeclareMathAlphabet{\mathbbold}{U}{bbold}{m}{n}
\newcommand{\zero}{\mathbbold{0}}
\newcommand{\unit}{\mathbbold{1}}
\newcommand{\inv}{\operatorname{inv}}
\newcommand{\Id}{\operatorname{Id}}
\newcommand\tran{{\operatorname{tran}}}
\newcommand{\supertropical}{\mathbb{ER}_{\max}}
\newcommand\T{{\operatorname{T}}}
\newcommand\sign{{\operatorname{sign}}}

\newcommand\adj{{\operatorname{adj}}}
\newcommand{\rmax}{\mathbb{R}_{\max}}
\newcommand{\R}{\mathbb{R}}
\newcommand{\cP}{\mathcal P}
\newcommand{\G}{\mathcal G}
\def\allperm{\mathfrak{S}}

\newcommand{\tr}{\operatorname{tr}}
\newcommand{\per}{\operatorname{per}}
\newcommand{\smax}{\mathbb{SR}_{\max}}
\def\SS{{\mathcal S}}

\newcommand{\cli}{\bar{i}}
\newcommand{\II}{\mathcal I}

\newcommand{\NEW}[1]{{\em #1\/}\index{#1}}
\newcommand{\thin}{thin }

\newcommand{\Mod}{\mathcal M} 
\def\TT{{\mathcal T}}

\newcommand{\skewproductstar}[2]{#1{\rtimes}{#2}}
\newcommand{\balance}{\,\nabla\, }
\newcommand{\succmod}{\;\mid\succeq^{\!\circ}\mid\;}
\newcommand{\summand}{signed product of signed elementary bijections\xspace}

\newtheorem{thm}{Theorem}[section]
\newtheorem{pro}[thm]{Proposition}
\newtheorem{lem}[thm]{Lemma}

\newtheorem{cor}[thm]{Corollary}

\newtheorem{prop}[thm]{Property}

\theoremstyle{definition}
\newtheorem{df}[thm]{Definition}

\theoremstyle{remark}
\newtheorem{rem}[thm]{Remark}
\newtheorem{exa}[thm]{Example}

\newcounter{myenumerate}
\renewcommand{\themyenumerate}{(\arabic{myenumerate})}
\newenvironment{myenumerate}{\begin{list}{\themyenumerate }{
\usecounter{myenumerate}\setlength{\labelsep}{0.5ex}
\setlength{\leftmargin}{0pt}\setlength{\labelwidth}{-\labelsep}
}}{\end{list}}

\title{Tropical compound matrix identities}

\author{Marianne Akian}
\address{Marianne Akian,
INRIA Saclay--\^Ile-de-France and CMAP, \'Ecole 
Polytechnique, CNRS. Address:
CMAP, \'Ecole Polytechnique,
Route de Saclay,
91128 Palaiseau Cedex, France.}
\email{Marianne.Akian@inria.fr}
\author{St\'ephane Gaubert}
\address{St\'ephane Gaubert,
INRIA Saclay--\^Ile-de-France and CMAP, \'Ecole 
Polytechnique, CNRS. Address: CMAP, \'Ecole Polytechnique,
Route de Saclay,
91128 Palaiseau Cedex, France.}
\email{Stephane.Gaubert@inria.fr}
\author{Adi Niv}
\address{Adi Niv,
Shamoon Academic College of Engineering.
Address: Mathematics Department, 56 Bialik St.
Beer-Sheva 84100, Israel.}
\email{Adini1@sce.ac.il}

\thanks{The first two authors have been partially supported by
the PGMO Program of FMJH and EDF, and by the MALTHY Project of the ANR Program.
This work was performed when the third author was with INRIA Saclay--\^Ile-de-France and CMAP, Ecole polytechnique, and it was supported by the French Chateaubriand grant and INRIA postdoctoral fellowship.}

\begin{document}
  
\begin{abstract} 
We prove identities on compound matrices  in extended tropical semirings. Such identities  
include analogues to properties of conjugate matrices, powers of matrices and~$\adj(A)\det(A)^{ -1}$, 
all of which have implications on the  eigenvalues of the corresponding matrices. A tropical 
Sylvester-Franke identity is provided as well.

\vskip 0.15 truecm

\vskip 0.1 truecm

\noindent \textit{AMSC: 15A15 (Primary), 15A09, 15A18, 15A24, 15A29, 15A75, 15A80, 15B99.} 	

\end{abstract}

\maketitle

\thispagestyle{empty}

\section{Introduction}

The max-plus or tropical semiring~$\mathbb{R}_{\max}$ is the set of real numbers~$\mathbb{R}$, 
completed with the element~$-\infty$, and equipped with the operations~$a\oplus b=\max\{a,b\}$ 
and~$a\odot b=a+b$ (also denoted as $ab$). The zero-element of this structure is therefore~$-\infty$, 
which is its minimal element. See for instance~\cite{BCOQ92,butkovicbook,MacStur,MPA} and the references therein for more background on linear algebra over the max-plus semiring.

The lack of additive inverses is a source of difficulties in the study of tropical structures. In particular, the notion of ``vanishing''
has to be adapted: a tropical polynomial vanishes
at a point if the maximum of the values of its monomials,
evaluated at this point, is achieved twice at least. 
In applications coming from real geometry~\cite{virodequantization},
one considers tropical polynomials
enriched with a sign information. Then, vanishing tropically
means that the maximum of the value of the monomials
with a positive tropical sign coincides with
the maximum of the value of the monomials
with a negative tropical sign. In this way, 
one can define the notion of polynomial identity
over the tropical semiring. 
Such polynomial identities
can often be proved by direct combinatorial
methods, i.e., by ``bijective proofs'',
along the lines of~\cite{straubing,zeilberger}
or of~\cite{gondran} (see also~\cite{GONDRAN1984147}). 
It was observed in~\cite{reutstraub} that 
certain determinantal identities over semirings can be derived
from their classical analogues, avoiding the recourse
to bijective proofs. This idea led to a {transfer principle}
in~\cite{G.Ths},
later refined in~\cite{LDTS}. As an application of the transfer
principle, a number of determinantal
identities (Laplace type expansions~\cite{reutstraub,LS,G.Ths,LDTS}, Binet-Cauchy theorem~\cite{G.Ths,gauBC96b,LDTS}) or
more advanced polynomial identities (Amitsur-Levitzki~\cite{gaubert96a,LDTS}),
were obtained, with several applications. Polynomial
identities have also appeared more recently in works on 
the ``supertropical'' extension of the tropical semiring~\cite{STMA,STMA2}. 

In the present paper, we establish tropical analogues of several
classical identities in the theory of determinants~\cite{ITD}. Unlike
the previously mentioned tropical determinantal identities, the identities that we establish have the remarkable feature that they do not follow from the transfer principle: more precisely, an application of the transfer principle would lead to weaker identities.

In order to formulate these identities, it is convenient to use 
the setting of {\em extensions} of semirings. This has
been developed in a number of works~\cite{LS,G.Ths,TA,LDTS,STMA,STMA2,AGG14}. Two basic extensions have been considered so far.
The {\em supertropical semiring}~\cite{TA,STMA,STMA2} is the union of two
copies of the set of tropical numbers, one copy represent
the ordinary numbers, whereas the other copy represent
``ghost'' numbers, encoding the fact that a maximum
is achieved twice. 
The  {\em symmetrized max-plus semiring}~\cite{LS,G.Ths,LDTS,AGG14} is the union
of three copies of the set of tropical numbers, representing
respectively tropically positive elements, tropically negative elements,
and ``balanced'' or singular elements, of the form
$a\oplus(\ominus a)$.
These two extensions may be thought of as special
hyperfields~\cite{krasner,connesconsani,1006.3034,baker}
(a hyperfield is a structure with a multivalued addition on a base set, 
the supertropical and symmetrized semirings can be identified
to the powerset semirings of two hyperfields). 
The supertropical numbers arise when considering images
of complex Puiseux series by the nonarchimedean valuation,
whereas symmetrized tropical numbers arise as images
of real Puiseux series.

We further elaborate on these different structures in Section~\ref{preliminaries}.  In particular, the obvious resemblance between them 
can be formalized thanks to the notion of {\em semiring with symmetry}~\cite{LDTS,AGG14}.
The latter are semirings equipped with an operation $a \mapsto \ominus a$, in which
singular elements (playing the role of the zero element) are of the form $a\ominus a$. In the supertropical case, the symmetry is just the identity
map, whereas the symmetry operation behaves formally as an ``opposite sign'' in the case of the symmetrized tropical semiring. Then, the notion of polynomial
identity has to be revised. Instead of looking for identities
of the form $a=b$, where $a,b$ can be polynomial expressions, we shall 
look for identities of the form $a\succeq^\circ b$, to be read
``$a$ {\em surpasses} $b$''.
The latter relation is defined by
$$a\succeq^\circ b\Leftrightarrow 
a=b\oplus (c\ominus c),\; \text{for some $c$.}$$
This way of writing identities may surprise at the first sight. However,
some of the most handy tropical polynomial identities 
are expressed in this way. The reader will not be wrong in imagining
that the presence of the singular term $c\ominus c$ 
accounts for the {\em irreversibility} of algebraic computations
in the tropical setting: when doing such computations, some
terms of the form $c\ominus c$, which vanish in the usual
algebra, remain in the tropical algebra.

Let us now come to our main topic.
We shall denote by~$A^{\wedge k}$ the~$k$th  
{\em compound matrix} or {\em Grassman power}
of~$A$, obtained by taking the~$k\times k$ minors of $A$ (see section~\ref{preliminaries}). 
The compound matrix has been widely studied. One can find definitions, identities and  basic algebraic properties in~\cite{MAHJ}.  
In contrast with the situation over rings,
the invertibility of the determinant does not imply
that a matrix is invertible.
Nevertheless, as shown in several works, especially~\cite{LS}, \cite{PI&CP},  the
familiar expression 
\[ A^\nabla=\adj(A)\det(A)^{ -1}
\enspace,
\]
which provides in classical algebra the inverse of an invertible matrix, 
can be defined as soon as the determinant of $A$ is invertible.
It does inherit classical properties, which justify the name of quasi-inverse.
Among others,  $A^\nabla$ can be factored as a product of elementary matrices~\cite{FTM}, a property which is equivalent to nonsingularity over fields, but not over semifields such as the tropical one.

In Sections~\ref{jacobi},  \ref{oid} and \ref{sf},
we use graph theory to provide tropical analogues for  identities 
concerning the compounds, quasi-inverse, powers, and so-called  conjugations
of matrices.
 These analogues can be interpreted  in terms perfect matchings: they are 
concerned with the existence of {different} permutations, described by the  {same subset of arcs}.

In Section~\ref{jacobi}, Theorem~\ref{QDF}, we prove  the analogue of 
Jacobi's identity 
$$\det(A)\big(DA^\nabla D\big)^{\wedge n-k}_{J^c,I^c}\succeq^\circ 
A^{\wedge k}_{_{I,J}},\text{ where $D$ is diagonal with $D_{i,i}=(\ominus\unit)^i,$}$$ 
which implies (Corollary~\ref{QEP})
$$\det(A)\tr
\big((A^{\nabla})^{\wedge n-k}\big)\succeq^\circ \tr\big(A^{\wedge k}\big)
\enspace .$$   
This establishes in particular~\cite[Conjecture 6.2]{PI&CP}.

Recall that~$\models$ denotes the specialization of the relation~$\succeq^\circ$ to the supertropical case. In 
Section~\ref{eppm} we use~$(A^{ m})^{\wedge k}\succeq^\circ(A^{\wedge k})^{ m}$  
in order to prove  the supertropical
 identity $$
\tr\big((A^{ m})^{\wedge k}\big)
\models \big(\tr(A^{\wedge k})\big)^{ m},\quad k=1,\ldots n$$  stated in 
Corollary~\ref{compowtr}. 
  The identity  $$E^{\wedge k}(E^\nabla)^{\wedge k}\succeq^\circ \mathcal{I}\ \ \ 
\text{(Proposition~\ref{nabcom})},$$  leads to Theorem~\ref{CEP},  concerning a so 
called tropical conjugation $$\tr\big((E^\nabla A E)^{\wedge k}\big)\succeq^\circ 
\tr(A^{\wedge k}).$$ 

Determinants have been studied in association to linear algebra, graph theory and algebraic geometry. One can 
see~\cite{ITD} for a survey on identities of determinantal identities,
including the
Sylvester--Franke identity.
In Section~\ref{sf}, we provide a tropical Sylvester-Franke identity,
which holds in particular over 
$\mathbb{R}_{\max}$ $$\per(A^{\wedge k})=\per(A)^{\left(\substack{n-1\\k-1}\right)}
\ \ \ \text{(Theorem~\ref{TSF})}.$$

The  characteristic polynomial of a square matrix is known for its applications in linear algebra. 
Yet, due to its connection with compound matrices, the tropical characteristic polynomial
 (see~\cite{TBE}   and ~\cite{ETCP}) has applications  in graph theory.  
We conclude this paper in Section~\ref{tcp}, by applying the identities of
Section~\ref{jacobi} and \ref{oid} to the tropical characteristic polynomials of the corresponding 
matrices. Theorem~\ref{CP} states
 \begin{equation}\label{cpint}f_M(X)\succeq^\circ \begin{cases}
\det(A)^{-1}X^n f_A(X^{-1})\enspace ,&M=A^\nabla\enspace,\\
f_A(x)\enspace ,& M =E^\nabla AE \enspace,
\end{cases}\end{equation} 
where $f_A$ denotes the characteristic polynomial of the matrix $A$,
and it also relates the  coefficients of the characteristic polynomials of~$A$ and 
its powers. We then  establish the connection to their corresponding eigenvalues, 
and deal with the special case of equality in Corollaries~\ref{CPNS}
and~\ref{cp1} respectively.

Note that the supertropical version of the  identities in~\eqref{cpint}  
were proved in~\cite{PCP}, \cite{PI&CP} and~\cite{shitov}.


\section{Preliminaries}\label{preliminaries}

Let~$\mathcal{P}(S)$ be the power set of a totally ordered set~$S$, and~$\mathcal{P}_k(S)
=\{I\in \cP(S):|I|=k\}$. We order~$\mathcal{P}_k(S)$ by  the (total) lexicographic order, 
with respect to the order in~$S$.

We denote by~$\allperm_{I,J}$ the set of  bijections between the two totally ordered 
sets~$I$ to~$J$. A bijection from a set to itself is called a \myemph{permutation}, 
and~$\allperm_{I}$ denotes the set of  permutations on~$I$. 
The permutation of $S$ such that $\pi(i)=i$, for all $i\in S$,
 is called the \myemph{identity permutation}, and denoted by $\Id$. 
Note that the restriction of a permutation is a bijection,  that is $$\pi|_J\in \allperm_{J,\pi[J]},
\ \forall J\in \cP(S)\; \text{and}\ \pi\in\allperm_{S}\enspace .$$ 
We denote~$\sigma[I]=\{\sigma(i):i\in I\}$, then,
for any $\sigma\in\allperm_{S}$ and $0\leq k\leq n$, we may define
$\sigma^{(k)}\in\allperm_{\mathcal{P}_k(S)}$, by $\sigma^{(k)}(I)=\sigma[I],\ \forall I
\in \mathcal{P}_k(S).$

The remaining part of this section is organized as follows.
In Section~\ref{gt} we recall basic notations in graph theory, followed by 
their matrix interpretation.
We then present in Section~\ref{sws} a unified algebraic setting, allowing
us to deal with several extensions of the tropical semiring.
In Section~\ref{tma} we provide matrix definitions, adjusted to the setting
and notation of Sections~\ref{gt} and~\ref{sws}.


\subsection{Graph theory}\label{gt}

We follow the terminology in ~\cite{AGT} and ~\cite{PA}. Let~$G$ be a weighted, 
directed graph (digraph), with~$n\in\mathbb{N}$ nodes, possible loops and no multiple arcs. 
The set of nodes is denoted by~$[n]=\{1,2,...,n\}$, an arc from~$i$ to~$j$ is denoted by~$(i,j)$, 
and its weight by~$a_{i,j}$. 
We shall assume here that the weights of the graph $G$ take their values in
a given semiring $(\SS,\oplus,\odot)$ with $\zero$ as its zero
(in particular $\zero$ is absorbing for $\odot$, 
see Section~\ref{sws}).
 The \myemph{weight matrix} of~$G$ is the~$n\times n$ matrix~$M_G$
over $\SS$, having~$a_{i,j}$ in its~$i,j$ position, when~$(i,j)$ is an arc of~$G$, and~$\zero$ 
otherwise. 
Conversely, if~$A=(a_{i,j})$ is an~$n\times n$ matrix with entries in $\SS$,
then the graph~$G$ of~$A$ is the  weighted, directed graph, with~$n$ nodes, and an arc 
from~$i$ to~$j$ with weight~$a_{i,j}$ if~$a_{i,j}\neq \zero$.

\begin{df} A \myemph{path} of length $m$  from $i$ to $j$ in $G$ is a set of $m$ arcs 
with concatenating nodes. This path is called \myemph{closed}~if $i=j$, \myemph{open}~if $i\neq j$,  \myemph{elementary}
 if  its intermediate nodes are distinct, and different from $i$ and $j$, and \myemph{maximal}  if it includes 
all the nodes~$[n]$. An elementary maximal path is also called \myemph{Hamiltonian}.
The \myemph{in-degree} (resp.~\myemph{out-degree}) of the node $i$, denoted by $d_{in}(i)$ 
(resp.~$d_{out}(i)$), is the number of arcs terminating (resp.~originating) in $i$.

 A \myemph{cycle} in~$G$ is an elementary closed path. 
A \myemph{bijection} (and in particular, a 
permutation) in~$G$ is a bijection~$\pi\in\allperm_{I,J}$, $I,J\in \cP_k([n])$, with a graph, $\G_\pi:=
\{(i,\pi(i)):\ i\in I\}$ composed of arcs in~$G$. In that case the subgraph of $G$ composed of the arcs 
in $\G_\pi$ satisfies
$$d_{out}(i)=d_{in}(j)=1,\ \forall i\in I,\  j\in J.$$
In the sequel, we shall identify a bijection with its graph.
The bijection~$\pi\in\allperm_{I,J}$ can be decomposed into disjoint elementary (open or closed) paths.
Denote by~$\cli\subseteq I$ the subset of nodes of~$I$ which 
are in the elementary path of~$i\in I$ in~$\pi$, and by~$C_\pi=\{\cli:\ i\in I\}$  the quotient set obtained by the partition 
of~$I$, induced from the elementary path decomposition of~$\pi$.
 In particular, a permutation~$\pi\in\allperm_{I}$ can be decomposed into disjoint cycles, and we shall 
identify as usual $\pi$ with the composition of these cycles. 
A bijection is called \myemph{elementary} if its decomposition has a single open elementary path and  trivial cycles  
 (loops).
\end{df}

We abuse these notations by using them for the matrix $M_G=(a_{i,j})$.   That is, the product 
of $m$ entries with concatenating indices describes a path. Notice that over $\mathbb{R}_{\max}$ 
the weight of a path is the sum of the weights of its arcs (or its entries).

 The number of non-$\zero$ entries with a right (resp.~ left) index $i$ is $d_{in}(i)$ (resp.~$d_{out}(i)$).

 The product $\bigodot_{i\in I} a_{i,\pi(i)}$ is the bijection $\pi\in \allperm_{I,J}$ of entries of $M_G$, 
and can  be decomposed into disjoint cycles and elementary open paths
\begin{equation}\label{cycfact}\bigodot_{\cli\in C_\pi}\bigodot_{j\in \cli}a_{j,\pi(j)},\ \text{ with }
\bigodot_{j\in \cli}a_{j,\pi(j)}=a_{i,\pi(i)}a_{\pi(i),\pi^2(i)}\cdots a_{\pi^{m_i-1}(i),\pi^{m_i}(i)},
\end{equation} 
 $$\text{where }\ \pi^{m_i}(i)\begin{cases}\in J\setminus I&\text{ if }\ \cli\not\subseteq J\text{  
(that is, }i\in I\setminus J),\\
&\qquad \text{obtaining an elementary open path,}\\=i&\text{ if }\ \cli\subseteq J\text{, 
obtaining a cycle.}\end{cases}$$
 In the special case of $I=J$, we get that $\pi^{m_i}(i)=i,\ \forall \cli\subseteq C_\pi.$ 

\begin{rem} \label{maxpath} If an elementary path has an intermediate index, it can be decomposed into 
two non-maximal elementary open paths, and a non-maximal elementary open path can  be extended at each of its ends into an 
elementary path.\end{rem}


\subsection{Semirings with a symmetry}\label{sws}
Recall that a \NEW{semiring} is a set $\SS$ with two binary operations, addition, denoted by $+$,
and multiplication, denoted by $\cdot$ or by concatenation, such 
that:
\begin{itemize}
\item $\SS$ is an abelian monoid under addition (with neutral element denoted by $\zero$ and called zero);
\item $\SS$ is a monoid under multiplication (with neutral element denoted
by $\unit$ and called unit);
\item multiplication is distributive over addition on both sides;
\item $s \zero=\zero s=\zero$ for all $s\in \SS$.
\end{itemize}

A semiring is \NEW{idempotent} when the addition is idempotent,
that is $a+ a=a$ for all $a\in \SS$.
It is \NEW{commutative} when the multiplication is commutative,
that is $ab=ba$ for all $a,b\in \SS$.
The max-plus semiring $\rmax$ described in the introduction is idempotent
and commutative.
Semimodules over semirings, and morphisms of semirings or semimodules are defined as for 
modules over rings, and morphisms of rings or modules, respectively.

\begin{df} A map $\tau:\SS\to \SS$ is a \NEW{symmetry} of the semiring $\SS$  if $\tau$ 
is a left and right $\SS$-semimodule homomorphism of order $2$, 
from  $\SS$ to itself, which means that it satisfies:
\begin{enumerate}
\item $\tau(a+ b)=\tau(a)+\tau(b),$

\item $\tau(\zero)=\zero,$

\item $\tau(ab)=a\tau(b)=\tau(a)b,$

\item $\tau(\tau(a))=a.$
\end{enumerate}
\end{df}

The typical example of a symmetry is the map $\tau(a)=- a$,
where $- a$ is the opposite of $a$ for the addition $+$ 
on a ring $\SS$, which satisfies also $a- a=\zero$.
Another possible symmetry, which works in any semiring is the
identity map $\tau(a)=a$.

The concept of semirings with symmetry first appeared in~\cite{G.Ths}, in order to prove some
 identities on tropical matrices  using a transfer principle.
It was then further used in~\cite{LDTS} in relation with the transfer principle and used 
in~\cite{LDTS,AGG14} to give a unified view of several extensions of the tropical semiring introduced
 in the literature, in particular the symmetrized max-plus semiring of~\cite{LS}, see also~\cite{G.Ths,BCOQ92}, 
and the Izhakian extension of the tropical semiring introduced in~\cite{TA}, see also~\cite{STLA}, that we next recall.
 Due to the lack of inverses in~$\mathbb{R}_{\max}$, its~$\zero$ often loses its algebraic role in 
the sense it is usually observed over rings. 
The following semirings were introduced to give an algebraic framework to the notions of matrix singularity, 
linear or algebraic equations and algebraic varieties over the tropical semiring.


\subsubsection{Symmetrized max-plus semiring}\label{sec-symmmax}
This extension, denoted~$\smax$, was studied in~\cite{LS,G.Ths,BCOQ92}.
Several equivalent constructions were proposed in~\cite{LDTS,AGG14}.
In particular, it can be seen as the union of three copies of~$\rmax$, denoted 
respectively~$\smax^\oplus=\rmax$,~$\smax^\ominus$ and~$\smax^\circ$,
in which the zero-elements~$-\infty$ are identified and denoted~$\zero$.
The copies of~$a\in\rmax$ in~$\smax^\oplus$,~$\smax^\ominus$ and~$\smax^\circ$
are respectively denoted~$\oplus a=a$,~$\ominus a$ and~$a^\circ$.
The set~$\smax$ is endowed with the operations~$\oplus$,~$\odot$ and~$\ominus$, 
such that~$(\smax,\oplus,\odot)$ is an idempotent semiring,
with the symmetry~$a\mapsto \ominus a,\; \smax\to \smax$.
The symmetry satisfies that 
$$\ominus (\oplus a)=\ominus a\in \smax^\ominus,\text{ so }\ominus (\ominus a)
=\oplus a\in \smax^\oplus,\ \ \forall a\in \rmax.$$
Moreover, 
$$a\ominus b:=(\oplus a)\oplus (\ominus b)=(\ominus b)\oplus (\oplus a)
=\begin{cases}a&\text{if }\ a>b\enspace,\\ \ominus b&\text{if }\ b>a\enspace,\\ 
a^\circ&\text{if }\ a=b\enspace,\end{cases}\ \ \ \forall a,b\in \rmax.$$
With these properties, and denoting~$a^\circ:=a\ominus a$ for all $a\in \smax$, 
one can show that~$\ominus (a^\circ)=a^\circ$, that~$\smax^\circ$ is an ideal 
of~$\smax$ and that the map~$a\mapsto a^\circ$  is a morphism of~$\rmax$-semimodules.

The balance relation~$\nabla$ on $\smax$ is defined as $$a\nabla b\Leftrightarrow 
a\ominus b\in \smax^\circ.$$ As a result~$a\nabla \zero$ if and only if 
$a\in \smax^\circ$. Such an element $a$ is said singular,
and equivalently it is non invertible in $\smax$.
Note that the symmetry $\tau:a\mapsto \ominus a$ in $\smax$ is a sort
of extension of the usual symmetry of a ring.
In general $a\ominus a=a^\circ \neq \zero$, but it is singular.

In~\cite{LDTS,AGG14}, the relation $\succeq^\circ$ is also introduced
as follows:
$$a\succeq^\circ b\Leftrightarrow 
\exists c\in \smax, \; a=b\oplus c^{\circ}\enspace .$$
Let $a\in\smax$. We denote by $|a|$ the element $b\in\smax^\oplus$ such that 
$a^\circ=b^\circ.$

The symmetrized max-plus semiring is useful to deal with systems of linear equations 
over~$\mathbb{R}_{\max}$.


\subsubsection{Supertropical  max-plus semiring} \label{stmps}
In  recent years, tropical geometry is in constant  development and became  
the main interest of algebraic geometry groups such as  Itenberg, Mikhalkin and 
Shustin (see for instance \cite{TAG} and~\cite{TGA}). 
 Izhakian's extension (see ~\cite{TA} and ~\cite{STLA}) was inspired by the  
observation from these groups that tropical varieties are the corner locus of some 
set of tropical polynomials.

 This extension is obtained by constructing a second copy of~$\mathbb{R}$, called the 
``ghost" ideal of this structure, denoted~$\mathbb{R}^\nu$.
Then, the {\em supertropical extension of the tropical semiring} is defined as
$$\supertropical
:= \mathbb{R}\cup \mathbb{R}^\nu\cup \{-\infty\}\enspace ,$$
and endowed  with operations~$\oplus$,~$\odot$ 
such that~$(\supertropical, \oplus,\odot)$ is a semiring,
and such that the copy of~$a\in \R$ in~$\mathbb{R}^\nu$, denoted~$a^\nu$,
satisfies~$a^\nu=a\oplus a$ and is considered to be singular.
That is,~$\supertropical$ is not idempotent.
However, the~$\oplus$ operation also satisfies~$a\oplus a\oplus a=a\oplus a=a^\nu$ 
(see below), which means that singularity is obtained when a maximum is attained at least twice.

The above properties imply that the injective map $a\mapsto a, \rmax\to \supertropical$ is not
 a morphism of additive monoids, so of semirings.
To further define the semiring operations, let us define
$a^\nu:=a$ when $a\in \mathbb{R}^\nu\cup \{-\infty\}$,
and apply the usual order of $\mathbb{R}\cup \{-\infty\}$
in its copy $\mathbb{R}^\nu\cup \{-\infty\}$.
Then, $\oplus$ and $\odot$ are such that the map $a\mapsto a, \rmax\to \supertropical$ is a 
morphism of multiplicative monoids, the map $a\mapsto a^\nu,\; \supertropical \to 
\mathbb{R}^\nu\cup \{-\infty\}$ is a surjective morphism of semirings, and of $\rmax$-semimodules,
and the addition is adjusted with respect to the ``ghostness" of the elements:
 $$a\oplus b=b\oplus a=\begin{cases}a&\text{if } a^\nu>b^\nu\enspace ,\\  b&\text{if } 
b^\nu>a^\nu\enspace ,\\ a^\nu&\text{if } a^\nu=b^\nu\enspace .\end{cases}$$

On $\supertropical$, one also defines the relation~$\models$ as $$a\models b\Leftrightarrow a=
b\oplus c^\nu\quad\text{for some } c\in \mathbb{R}\cup \{-\infty\},$$  and in particular~$a^\nu
\models \zero,\ \forall a\in \supertropical$.

Consider on~$\supertropical$ the identity symmetry~$\tau(a)=a$. 
If the semiring was idempotent,
it would satisfy~$a\oplus \tau(a)=a$, which cannot be set as a singular element.
However, in~$\supertropical$, 
we have that~$a\oplus\tau(a)=a^\nu$ is in general non-zero but
is singular.

Let $a\in\supertropical$. We denote by $|a|$ the element $b\in\mathbb{R}_{\max}$ 
such that $a^\nu=b^\nu.$

The supertropical max-plus semiring  contributes to the understanding of tropical roots 
over~$\mathbb{R}_{\max}$ in the sense of tropical geometry.


\subsubsection{Unified theory}\label{sec-uni}

The above semirings are particular examples of semirings
with a symmetry. They enjoy some additional properties that are necessary 
to obtain the identities presented in  Sections~\ref{jacobi}and~\ref{oid}.
We shall present here a unified construction already presented 
in~\cite{AGG14}, which contains more examples.

 Let $\SS$ be a commutative semiring with a symmetry.
Let us denote by $\oplus$ its addition, by $\odot$ its multiplication,
by $\zero$ its zero, by $\unit$ its unit, and by $\ominus$ its
symmetry (which means that we write $\ominus a$ instead of $\tau(a)$).
Throughout, we use $a\ominus b$ for $a\oplus(\ominus b)$, and set $a^\circ:=
a\ominus a$. This implies that  $\ominus a^\circ=a^\circ$.
These notations coincide with the ones of Section~\ref{sec-symmmax}
when $\SS$ is the symmetrized max-plus semiring $\smax$.
When $\SS$ is the supertropical semiring $\supertropical$,
$\ominus$ has to be understood as $\oplus$ and $a^\circ$ coincides with
$a^\nu$. 
In the above semirings, an element $a^\circ$ is singular and non invertible,
but it is not necessarily~$\zero$. To generalize this, we consider 
in $\SS$, the subset 
\[
\SS^\circ := \{a^\circ \mid a\in \SS\}  \enspace .
\]

Let us denote by $\SS^*$ the set of invertible elements of $\SS$.
The set $\SS^\circ$ is an ideal,
hence, either  $\SS^\circ=\SS$ or $\SS^\circ$
contains no invertible element of $\SS$ ($\SS^*\subset \SS\setminus \SS^\circ$).
If $\SS$ is a totally ordered idempotent semiring, like $\rmax$, 
then the only symmetry on $\SS$ is the identity
(see~\cite[Prop.\ 2.11]{AGG14}), then $\SS^\circ= \SS$.
For any semiring $\SS$, we shall also consider the subset 
\[ \SS^\vee:=\SS^*\cup \{\zero\}\enspace .\]
Note that in~\cite{AGG14}, this notation was used,
when $\SS^\circ\neq \SS$, to denote any set 
such that $\SS^*\cup \{\zero\}\subset 
\SS^\vee \subset (\SS\setminus \SS^\circ )\cup\{\zero\}$,
the elements of which were called the \NEW{\thin} elements.
Here we restrict the value of $\SS^\vee$ while putting no constraint on 
$\SS$.
In the sequel, we shall say that an element of $\SS$ is nonsingular if 
it belongs to $\SS^*$, and that it is singular otherwise.
Then, if $\SS^\circ\neq \SS$, the elements of $\SS^\circ$ are necessarily
singular, and if $\SS$ equals $\supertropical$ or $\smax$,
these are the only ones.
Note however that if $\SS=\rmax$, $\zero$ is the only singular element.

In the sequel, we shall consider particularly semirings $\SS$
with a symmetry which are either totally ordered idempotent semifields
(that is such that $\SS\setminus\{\zero\}=\SS^*$ is a totally ordered group
and so $\SS=\SS^\vee=\SS^\circ$)
or semirings satisfying $\SS^\circ\neq \SS$ (so
$\SS^\vee \subset (\SS\setminus \SS^\circ )\cup\{\zero\}$),
and some additional
properties described below.

\begin{df}\label{orderdf} On any semiring $\SS$ with a symmetry, one defines the relations:
\[ a\preccurlyeq b \iff b\succcurlyeq a\iff b=a\oplus c \text{ for some } c\in \SS\enspace,\]
\[ a\preccurlyeq^\circ b \iff b\succcurlyeq^\circ a\iff b=a\oplus c \text{ for some } 
c\in \SS^\circ\enspace,\]
and 
\[ a\curlyeqprec b \iff b\curlyeqsucc a\iff b=a\oplus b \enspace.\]
\end{df}
The relations~$\preccurlyeq,\preccurlyeq^\circ$ are preorders (reflexive and transitive),
 compatible with the laws of~$\SS$. They may not be antisymmetric. 
The relation~$\curlyeqprec$ is  antisymmetric and transitive, compatible with the 
laws of~$\SS$ (if~$a\curlyeqprec b$, then~$ac\curlyeqprec bc$ ; if~$a\curlyeqprec b$ 
and~$a'\curlyeqprec b'$, then~$a\oplus a'\curlyeqprec b\oplus b'$). It is reflexive when~$\SS$ is 
idempotent. Obviously, the relation~$\curlyeqprec$ implies the relation~$\preccurlyeq$.

Note that both~$\preccurlyeq$ and~$\curlyeqprec$ are such that
all elements of~$\SS$ are nonnegative, that is~$a\preccurlyeq\zero$
and~$a\curlyeqprec\zero$ for all~$a\in \SS$.

We apply these relations to matrices (and in particular to vectors) entry-wise, 
and to polynomials coefficient-wise.

\begin{df} 
A semiring~$\SS$ is said to be \NEW{naturally ordered} when~$\preccurlyeq$ 
(or equivalently~$\succcurlyeq$) is an order relation, and 
in that case~$\preccurlyeq$ is called the \NEW{natural order} on~$\SS$, 
and~$\succcurlyeq$ is its opposite order. 
\end{df}
When $\preccurlyeq$ (or $\succcurlyeq$) is an order relation, so are 
$\preccurlyeq^\circ$ and $\succcurlyeq^\circ$, and their extensions to matrices 
and polynomials.
When~$\SS$ is an idempotent semiring,~$\preccurlyeq$ is equal to~$\curlyeqprec$,~$\SS$ 
is necessarily naturally ordered, and 
a naturally ordered semiring is necessarily zero-sum free.

\begin{df}
When $\SS$ is naturally ordered and $(a_n)_{n\geq 0}$ is a nondecreasing 
(resp.\ nonincreasing) sequence
of scalars, matrices or polynomials over $\SS$, we say that
the sequence converges towards $a$, if the supremum (resp.\ infimum)
of the $a_n$ exists and is equal to $a$.
\end{df}

\begin{df}
Let $\SS$ be a semiring and $\Mod$ be a totally ordered idempotent semiring.
We say that a map $\mu:\SS\to \Mod$ is a \NEW{modulus} if it is a 
surjective morphism of semirings.
In this case, we denote $\mu(a)$ by $|a|$ for all $a\in \SS$,
and for $a,b\in \SS$, we say that $b$ dominate $a$ when $|a|\preceq |b|$. 
\end{df}

The absolute value is applied to matrices (and in particular to vectors) entry-wise, 
and to polynomials coefficient-wise.
A modulus on a semiring $\SS$ with a symmetry  satisfies necessarily 
 $|\ominus a|=|a|$ and so $ |a^\circ|=|a|$ for all $a\in \SS$
(see~\cite[Prop.\ 2.11]{AGG14}).
In particular, $\mu=\mu|_{\SS^\circ}\circ c$ where $c:\SS\to\SS^\circ$ is the
map $a\mapsto a^\circ$.
If $\SS^\circ$ is already idempotent and totally ordered, then 
the map $c$ is a modulus. This is the case when
$\SS=\smax$ or $\SS=\supertropical$,
and indeed in these cases, one considered a modulus map 
with values in $\rmax$, such that $\mu|_{\SS^\circ}$ is an isomorphism.
This is also the case when $\SS$ is already an 
idempotent and totally ordered semiring, in which case
$c$ is the identity.
On these semirings, 
we have the following three properties. Moreover,
Property~\ref{prop1} holds on any
semiring of the form $\skewproductstar{\SS'}{\Mod}$ as in~\cite{AGG14}.
\begin{prop}\label{prop1} $\text{If }\ a,b\in\SS\text{ such that }\ 
|a|\prec|b|\;\Rightarrow\; a\oplus b=b$.
\end{prop}

\begin{prop}\label{prp1}
$\text{If }\ a,b\in\SS\text{ such that }b\succeq a,\ |a|=|b|\ \text{and }\ b\in \SS^{\vee}
\text{, then }a=b\enspace.$
\end{prop}

\begin{cor}\label{cormoins}
 $\text{If }\ a,b\in\SS\text{ such that }a\curlyeqprec \ominus b\ \text{and }\ b\in \SS^{\vee}\text{, then }a=\ominus b
\text{ or }a\curlyeqprec  b.$\end{cor}
\begin{proof}From Property~\ref{prop1}, if~$|a|\prec|\ominus b|=|b|$, then~$a\curlyeqprec b.$ Since $\ominus \unit\in \SS^*$,  $b\in \SS^{\vee}$ implies $\ominus b\in \SS^{\vee}$. Then, since 
$a\curlyeqprec \ominus b$ implies $a\preccurlyeq \ominus b$, we get from Property~\ref{prp1} that
 if~$|a|=|\ominus b|$, then~$a= \ominus b.$\end{proof}

The following consequence will also be useful.
\begin{pro}\label{prop2} If $\SS$ satisfies $ \SS^\circ\neq \SS$ and 
Property~\ref{prp1}, then 
\[a,b\in\SS\text{ such that }\ b\succeq^\circ a\ \text{and }\ b\in \SS^{\vee}\;
\text{implies }\; a=b\enspace .\]
\end{pro}
\begin{proof}
Assume $b\succeq^\circ a$
and $b\in \SS^{\vee}$. Then, $a\preceq b$ and $|a|\preceq |b|$. 
If $|a|= |b|$, then using Property~\ref{prp1}, 
and $b\in \SS^{\vee}$, we get $a=b$. 
Otherwise $|a|\prec|b|$, so $b\neq \zero$ and since $b=a\oplus c^\circ$, we get 
$|b|=|c^\circ|$. Then, due to Property~\ref{prp1}, we obtain
$c^\circ=b\in \SS^\circ\cap \SS^{\vee}\setminus\{\zero\}=\emptyset$, 
a contradiction.
\end{proof}

\begin{rem}
In \cite{AGG14}, some different properties were considered, which imply
the above ones.
For instance, $\SS^\vee$ satisfies Property 5.2 of  \cite{AGG14} if 
\begin{equation}\label{5.2}
 (x\in \SS^{\vee}\text{ and } x\preceq y)
\implies\;\exists z\in \SS^{\vee} \text{ such that }
x\preceq z\preceq y,\; z \balance y,\;\text{and}\; |z|=|y|\enspace, \end{equation}
and it satisfies Property 5.4 of  \cite{AGG14} if
\begin{equation}\label{5.4}
 ( x,y\in \SS^{\vee},\;  x \preceq  y \text{ and } |x|=|y|) \implies x=y
\enspace. \end{equation}
When $\SS$ is naturally ordered, these two properties imply Property~\ref{prp1}.
Indeed, if $b\succeq a$, $|a|=|b|$, and $b\in \SS^{\vee}$, 
from~\eqref{5.2}, $\exists c\in  \SS^{\vee}\text{ s.t.~}c\preceq a\preceq b,\ c\nabla a,\ 
|a|=|b|=|c|$. Then, since $b,c \in\SS^{\vee}$,
from~\eqref{5.4} we obtain $c=b$ and therefore $a=b$.
\end{rem}

The following property will also be needed.

\begin{prop}\label{propnew} $\SS^\circ$ is idempotent, that is
for all $a\in \SS$, $a^\circ \oplus a^\circ=a^\circ$.
\end{prop}
\begin{lem}\label{lem-propnew} Property~\ref{propnew} is equivalent to
the condition: $a\oplus a^\circ=a^\circ$, for all $a\in \SS$. It implies 
 $\ominus a\oplus a^\circ$, for all $a\in \SS$.
\end{lem}
\begin{proof}
Since $a^\circ\preccurlyeq  a\oplus  a^\circ\preccurlyeq  a^\circ\oplus  a^\circ$,
Property~\ref{propnew} implies that, for all $a\in \SS$,
$a\oplus a^\circ=a^\circ$. Since $(\ominus a)^\circ=a^\circ$, the latter property
implies that, for all $a\in \SS$, $\ominus a\oplus a^\circ$.
Conversely, if,  for all $a\in \SS$,
$a\oplus a^\circ=a^\circ$, then,  for all $a\in \SS$, 
$\ominus a\oplus a^\circ=a^\circ$, and so 
$a^\circ\oplus  a^\circ= a\oplus (\ominus a)\oplus a^\circ= a\oplus  a^\circ=a^\circ$.
\end{proof}

\subsection{Tropical matrix algebra}\label{tma}
Throughout the following sections, we shall consider a
commutative semiring with a symmetry, denoted by $\SS$,
which may have some additional properties, like having a modulus.
To simplify the presentation, we shall say that $\SS=\TT$,
and denote $\SS$ by $\TT$, if $\SS$ is naturally ordered,
has a modulus taking its values in a
totally ordered idempotent semifield $\Mod$, and satisfies
Properties~\ref{prop1}, \ref{prp1} and~\ref{propnew} (although the
latter property is not needed in the present  section).
Similarly, we shall say that $\SS=\Mod$, and denote $\SS$ by 
$\Mod$, if $\SS$ is a totally ordered idempotent semifield
with its (unique) identity symmetry and the identity modulus.
Following the notations of the previous section, we formulate some  basic 
definitions. One may also find in~\cite{MA} further combinatorial motivation for 
the objects discussed.

\begin{df}
The \myemph{trace} of $A=(a_{i,j})\in \SS^{n\times n}$ is defined as $$\tr(A)=
\bigoplus_{i\in[n]} a_{i,i},$$ and if $\SS$ has a modulus,  we refer to any diagonal 
entry  with highest absolute value as \myemph{a dominant diagonal entry}.
\end{df}

\begin{rem}\label{comtr} Denote $B=(b_{i,j})$. As usual, $\tr(AB)=\tr(BA),$ since 
$$\bigoplus_{i\in[n]}\bigoplus_{t\in[n]}a_{i,t}b_{t,i}=\bigoplus_{t\in[n]}\bigoplus_{i\in[n]}b_{t,i}a_{i,t}.$$  
\end{rem}

\begin{df}\label{EP}  Let $\SS$ be fixed.
The \myemph{sign} of a bijection~$\sigma\in \allperm_{I,J}$ is~$\sign(\sigma)=
(\ominus\unit)^{|\inv(\sigma)|}\in \SS,$ where   $$\inv(\sigma)=\{(i,j)\in I^2:
\ i<j\ \text{and}\ \sigma(i)>\sigma(j)\}$$ is the set of inversions in~$\sigma,$  
taken with respect to the orders in~$I$ and~$J$.

Similarly to~\eqref{cycfact}, we call \myemph{signed permutation (resp.~signed bijection)}
 of a permutation~$\pi$ of~$I$ (resp.~a bijection~$\pi:I\rightarrow J$) the expression~$\sign(\pi)
\bigodot_{i\in I} a_{i,\pi(i)}.$
Recall that a cycle is a permutation on the set of its  indices (resp.~an elementary open path is a 
bijection from the set of its left indices to the set of its right indices). It is well-known that a 
signed permutation is the product of its signed cycles. However, a
signed bijection is not the product of its signed cycles and signed elementary open paths (e.g.~$\sign((1\ 4)(2))\ne\sign(1\ 4)\sign(2)$). 
Note that a path can be decomposed into  an elementary path of the same source and 
target and  (not necessarily disjoint) cycles,  starting and ending at the points 
of repeating indices. 

\end{df}

It is a well known fact that~$(\ominus\unit)^{|\inv(\sigma)|}=(\ominus\unit)^{|\tran(\sigma)|},$ 
for all $\sigma\in \allperm_{I}$, where $\tran(\sigma)$ denotes the number of transpositions  
which the permutation is factored into.

\begin{prop}\label{signs}
Let~$\sigma\in\allperm_{I, J},\ I,J\subseteq S$.  For~$\pi\in \allperm_{S}$, 
the unique bijection~$\rho\in\allperm_{\pi[I], \sigma[I]}$ such that~$\rho\circ\pi(i)=\sigma(i),\ 
\forall i\in I,$ satisfy~$\ \sign(\sigma)=\sign(\pi|_I)\sign(\rho).$

\end{prop}
This property is a  generalization of the multiplicativity of permutation sign, proved analogously by  
taking the permutations on~$\big[|I|\big]$ induced from~$\sigma,$~$\pi|_I$ and~$\rho$.

\begin{df}
We define the \myemph{determinant} of a  matrix $A=(a_{i,j})\in \SS^{n\times n}$ 
to be 
$$\det(A)=\bigoplus_{\sigma \in \allperm_{[n]}}\sign(\sigma)\bigodot_{i\in[n]} a_{i,\sigma(i)}.$$ 
We define a matrix  to be  \myemph{nonsingular}  if~$\det(A)\in \SS^*$,
and if $\SS$ has a modulus,  we refer to any permutation with weight of highest absolute value  as 
\myemph{a dominant permutation}.
\end{df}
When $\SS=\Mod$, a totally ordered idempotent semifield,
or~$\SS=\supertropical$, we have $\sign(\sigma)=\unit$, so 
the determinant is actually the same as the \textit{permanent}
$$\per(A)=\bigoplus_{\sigma \in \allperm_{[n]}}\bigodot_{i\in[n]} a_{i,\sigma(i)}.$$ 
For a general $\SS$ with a modulus, the modulus of the determinant coincides with the permanent 
of the modulus of $A$ over $\Mod$.
Moreover, when $\Mod=\rmax$, this is the \textit{max permanent}  of the modulus of $A$
 defined in~\cite{CCP} and~\cite{MNMX}, and it is the value of
an associated optimal assignment problem. 

Note that when $\SS=\Mod$,
and thus when $\SS=\rmax$, a matrix is nonsingular in the above sense 
if and only if $\per(A)\neq \zero$ that is if the
optimal assignment problem is feasible.
So the nonsingularity of a matrix $A\in \rmax^{n\times n}$ seen as a matrix
over $\rmax$ does not imply its tropical nonsingularity
in the  sense of~\cite{RGST}.
However, using the injection of $\rmax$ into $\supertropical$,
$A$ can also be seen as a matrix over $\supertropical$, that is as
an element of $\supertropical^{n\times n}$, and in that case
$A$ is nonsingular if and only if $A$ is tropically nonsingular
in the  sense of~\cite{RGST}.

\begin{df} Let $\SS$ be fixed.
 The~$n\times n$ matrix with $\unit$ on the diagonal and $\zero $ otherwise is the  \myemph{identity matrix},  
denoted by $\mathcal{I}$, or by $\mathcal{I}_n$ if indicating its size is required.

A matrix $A$ is \myemph{ invertible} if there exists a matrix~$B$ such that $$A B=B A=\mathcal{I}.$$
Let~$S=\{s_i:\ i\in[n]\}$ be a totally ordered set of cardinality $n$ with~$s_i<s_j\ \forall i<j$.
With an abuse of notation, we shall consider matrices indexed by the ordered elements of $S$, called $S\times S$ 
matrices, and identify them to $n\times n$ matrices. Then, a $S\times S$ matrix $A$ with entries in $\SS$
is identified to an element of $\SS^{n\times n}$ and its entries will be denoted either by $A_{s,t}$ with $s,t\in S$,
or by $a_{i,j}$ with $i, j, \in [n]$, and $a_{i,j}=A_{s_i,s_j}$.
In particular, when $S=[n]$, $a_{i,j}=A_{i,j}$.

 A $S\times S$ matrix~$A=(a_{i,j})$ is defined to be the \myemph{permutation matrix} associated to the 
permutation $\pi\in \allperm_{S}$, and will be denoted $P_{\pi}$, if, for all $i,j \in [n]$, 
$$ a_{i,j}=A_{s_i,s_j}=
\begin{cases}
\unit&,\; \text{if}\; s_j=\pi(s_i)\\
\zero&,\; \text{otherwise.} 
\end{cases}$$ 
A $S\times S$ matrix~$A=(a_{i,j})$ is defined to be the \myemph{diagonal matrix} with diagonal 
entries given by the sequence
$q=(q_1,\ldots, q_n)$ of $\SS$ or the map $q:S\rightarrow\SS,\; s_i\mapsto q_i$, and will be 
denoted $D_q$ , if, for all $i,j \in [n]$, 
$$a_{i,j}=A_{s_i,s_j}=
\begin{cases}
q_i&,\; \text{if}\; j=i\\
\zero&,\; \text{otherwise.} 
\end{cases}$$ 
When  $q$ is such that $q_i\in\SS^*\; \forall i\in [n]$, we denote $q^{-1}$ the sequence
or map such that $(q^{-1})_i=(q_i)^{-1}$, for all  $i\in [n]$.
\end{df}

\begin{rem}\label{inv}(see ~\cite{IBM})
Let $\SS$ be naturally ordered. It is necessarily zero-sum free, so
 a  matrix~$A$ is invertible in~$\SS^{n\times n}$ if and only if it is  the product of a permutation 
matrix~$P_{\pi}$, and a diagonal matrix~$D_q$ with an invertible determinant. We 
define~$M=D_q P_\pi $ as a \myemph{generalized permutation matrix}, also known as \myemph{monomial}. 
In particular~$P_{\pi}^{-1}=P_{\pi^{-1}}$, $D_{q}^{-1}=D_{q^{-1}}$, $M=P_\pi D_{q\circ \pi^{-1}}$,
and $M^{-1}=D_{q^{-1}}P_{\pi^{-1}}$.
\end{rem}

\begin{pro}\label{resign}
Let~$\pi\in \allperm_{[n]}$,~$I\in \mathcal{P}_k([n])$. We define~$\tau_I\in \allperm_{[n]}$ to be the 
permutation sending $[k]$ to $I$, and such that its restrictions to $[k]$ and $[n]\setminus[k]$ are order-preserving.
We define $\tau_{\pi[I]}$ similarly. Then,

\begin{enumerate}

\item\label{resign-prop1} $\sign(\tau_I)\sign(\tau_{\pi[I]})=(\ominus\unit)^{\sum_{i\in I}i+\pi(i)}$.

\item\label{resign-prop2} $\sign(\pi)=
\sign(\pi|_I)\sign(\pi|_{I^c})(\ominus\unit)^{\sum_{i\in I}i+\pi(i)}.$

\end{enumerate} \end{pro}

\begin{proof} To show~\eqref{resign-prop1}, we note that the number of transpositions in~$\tau_I$ and~$\tau_{\pi[I]}$  are    
$$\sum_{j\in[k]}i_j-j\ \ \text{ and }\ \ \sum_{j\in[k]}l_j-j=\sum_{j\in[k]}\pi(i_j)-j,$$ respectively, where $I=
\{i_1<\dots <i_k\},\ \pi[I]=\{l_1<\dots <l_k\}$. As a result $$\sign(\tau_I)\sign(\tau_{\pi[I]})=
(\ominus\unit)^{\sum_{i_j\in I}i_j+\pi(i_j)-2j}=(\ominus\unit)^{\sum_{i\in I}i+\pi(i)}.$$
For~\eqref{resign-prop2}, The permutation matrix~$P_{\tau_{\pi[I]}^{-1}\pi\tau_I}=P_{\tau_I}P_{\pi}P_{\tau_{\pi[I]}^{-1}}$ 
has the~$I\times \pi[I]$ block of~$P_\pi$ as its~$[k]\times [k]$ block, and the~$I^c\times \pi[I^c]$ block 
of~$P_\pi$ as its~$([n]\setminus[k])\times ([n]\setminus[k])$ block. The permutation matrix of~$\tau_I$ 
(resp.\ $\tau_{\pi[I]}$) has~$\mathcal{I}_k$ as its~$[k]\times I$ (resp.~$[k]\times \pi[I]$) block, and 
has~$\mathcal{I}_{n-k}$ as its~$([n]\setminus[k])\times I^c$ (resp.~$([n]\setminus[k])\times \pi[I^c]$) 
block.
Therefore  
$$\sign(\tau_{\pi[I]}^{-1})\sign(\pi)\sign(\tau_I)=\sign(\tau_{\pi[I]}^{-1}\pi\tau_I)=$$
$$\big[\sign(\tau_{\pi[I]}|_{[k]})^{-1}\sign(\pi|_I)\sign(\tau_I|_{[k]})\big]\big[\sign(\tau_{\pi[I]}|_{[n]
\setminus[k]})^{-1}\sign(\pi|_{I^c})\sign(\tau_I|_{[n]\setminus[k]})\big].$$
Since~$\tau_{\pi[I]}$ and~$\tau_I$ have no inversions over~$[k]$ and~$[n]\setminus[k]$, their signs 
are~$\unit$. As a result
  $$\sign(\pi)=\sign(\tau_{\pi[I]})\sign(\pi|_I)\sign(\pi|_{I^c})\sign(\tau_I)^{-1}=
\sign(\pi|_I)\sign(\pi|_{I^c})(\ominus\unit)^{\sum_{i\in I}i+\pi(i)}\enspace .\qed$$
\renewcommand{\qedsymbol}{}
\end{proof}

The following corollary is a result of Property~\ref{signs} and Proposition~\ref{resign}. 
\begin{cor}\label{eps}
For an elementary path~$\rho$  from~$i$ to~$j$, we consider the cycle~$\sigma=(i\ \rho(i)\ \rho^2(i)\cdots j)\in\allperm_{[n]}$.  If~$\rho$ is the concatenation $\rho_1\rho_2$ of elementary paths $\rho_1,\rho_2$ from~$i$ to~$k$ and from~$k$ to~$j$ respectively, $\sigma_1=(i\ \rho_1(i)\ \rho_1^2(i)\ \dots\ k)\in\allperm_{[n]},\ \sigma_2=(k\ \rho_2(k)\ \rho_2^2(k)\ \dots\ j)\in\allperm_{[n]}$ are the corresponding cycles, and~$\ell$ denotes the length of~$\rho$,  then  $$\sigma|_{\{j\}^c}=\sigma_1|_{\{k\}^c}\circ\sigma_2|_{\{j\}^c}\ \text{ and }\ \sign(\sigma|_{\{j\}^c})=(\ominus\unit)^\ell (\ominus\unit)^{i+j}=\sign(\sigma_1|_{\{k\}^c})\sign(\sigma_2|_{\{j\}^c}).$$ \end{cor}

\begin{proof} We get~$\sigma|_{\{j\}^c}=\sigma_1|_{\{k\}^c}\circ\sigma_2|_{\{j\}^c}$, since~$\rho_1\rho_2$ is an elementary path. Then, from Proposition~\ref{resign}, we get
\begin{eqnarray*}\sign(\sigma|_{\{j\}^c})&=&\sign(\sigma)\sign(\sigma|_{\{j\}} )(\ominus\unit)^{i+j}\\
&=&(\ominus\unit)^\ell (\ominus\unit)^{i+j}\\
&=&\sign(\sigma_1|_{\{k\}^c})\sign(\sigma_2|_{\{j\}^c})\enspace. \qed\end{eqnarray*}
\renewcommand{\qedsymbol}{}
\end{proof}

\begin{thm}\label{detAB} 
For $A,B\in\SS^{n\times n}$, we have that \begin{equation}\label{deter}\det(A B) 
\succeq^\circ \det(A)\det(B),\end{equation} with equality when 
\begin{enumerate}
\item \label{detAB1} 
$\SS$ is naturally ordered and  $A$ or $B$ are invertible, 
\item  \label{detAB2} $\SS=\TT\ne\TT^\circ$ and $\det(AB)\in\TT^\vee$.
\end{enumerate}
\end{thm}

Identity~\eqref{deter} was proved in~\cite[Proposition~2.1.7]{G.Ths}.
It subsequently appeared in~\cite[Lemma 3.2]{gauBC96b}, with an application
to minimal realization of linear recurrent sequences. 
It also appeared in~\cite[Theorem~3.5]{STMA} over~$\supertropical$.
Statement~\eqref{detAB1} follows from~\cite[Proposition~3.4]{AGG14} 
and the property that an invertible matrix is necessarily a  monomial matrix.  \eqref{detAB2} is an immediate consequence of 
Proposition~\ref{prop2}. 

\begin{df} A  matrix~$A=(a_{i,j})\in \SS^{n\times n}$ is  \myemph{definite} if~ $\det(A)=a_{i,i}=
\unit\ \forall i\in[n].$ \end{df}
Obviously, a definite matrix is nonsingular (over $\SS$).
When $\SS=\rmax$, the above definition is equivalent to the one of~\cite{MA}
of definite matrices.
However, when $A\in \rmax^{n\times n}$ is seen as a matrix over $\supertropical$,
$\det(A)=\unit$ implies that the optimal assignment problem has a unique
solution, hence $A$ is definite over $\supertropical$ if and only if
$A$ is \textit{strictly definite} in the sense of~\cite{MA}.
Note also that definite matrices have to be distinguished from \textit{normal matrices}, defined 
in~\cite{MA} to have non-positive (that is $\preceq \unit$) non-diagonal entries, that will not be in 
use in the present paper.

When $\SS=\TT$, and $A$ is nonsingular, then a  dominant permutation of~$A$  
has a weight equal to the determinant of $A$ (due to Property~\ref{prp1}).
Then, this dominant permutation may be normalized and relocated to the diagonal, using  an invertible
 matrix, obtaining a definite matrix with a dominant normalized~$\Id$-permutation. That is,~$A=P\bar{A}
\text{ or }\bar{A}P,$ where~$P$ is an invertible matrix such that $\det(A)=\det(P)$, and the
 matrix~$\bar{A}$ is definite. We then use the following definition.

\begin{df} A definite matrix~$\bar{A}$  is  a \myemph{left} (resp.\ \myemph{right}) \myemph{definite form}  of 
a nonsingular matrix~$A\in \TT^{n\times n}$,  if~$P$ {normalizes} it by acting on its rows (resp.\ columns):
 $A=P\bar{A}$ (resp.\ $A=\bar{A}P$). The invertible matrix~$P$ is the \myemph{left} (resp.~\myemph{right}) 
\myemph{normalizer}  of~$A$, corresponding to its  \myemph{left} (resp.~\myemph{right}) \myemph{definite form}.\end{df}

We introduce a standard combinatorial property of definite matrices.
\begin{lem}\label{dfs} Let~$A$ be an~$n\times n$ definite matrix.  
\begin{enumerate}
\item\label{dfs1} For every signed  permutation~$\sigma\in\allperm_I\setminus\{\Id\},\ I\in\mathcal{P}_k([n])$, in~$A$
 we have$$\sign (\sigma)\bigodot_{i\in I}
A_{i,\sigma(i)}\curlyeqprec\unit.$$ 

\item\label{dfs2} The signed  cycle  of every $\cli\in C_\sigma$ s.t.~$\cli\ne \{i\}$ satisfies $$ (\ominus\unit)^{|\cli|-1}
\bigodot_{j\in \cli}A_{j,\sigma(j)}\curlyeqprec\unit.$$

\item\label{dfs3} 
Consider a signed  bijection~$\sigma\in\allperm_{\{t\}^c,\{s\}^c}$ in~$A$,
with  $s,t\in[n]:\ s\ne t$. Let
$\bar{s}\in C_\sigma$ and let $\tau\in\allperm_{\{t\}^c,\{s\}^c}$ be  such that $\tau|_{\bar{s}}=\sigma|_{\bar{s}}$ and $\tau|_{(\bar{s}\cup \{ t\})^c}=\Id$. We have
$$\sign(\sigma)\bigodot_{\bar{i}\in C_\sigma}\bigodot_{j\in\bar{i}}A_{j,\sigma(j)}\curlyeqprec
\sign(\tau)
\bigodot_{j\in\bar{s}}A_{j,\tau(j)}.$$ 
\end{enumerate}
\end{lem}

\begin{proof}Let $\sigma\ne \Id$ be a permutation. Using Definition~\ref{orderdf} 
$$\unit\preccurlyeq\unit \oplus\sign(\sigma)\bigodot_{i\in [n]} A_{i,\sigma(i)}\preccurlyeq\unit \oplus
\bigoplus_{\sigma\ne \Id}\sign(\sigma)\bigodot_{i\in [n]} A_{i,\sigma(i)}=\unit,$$ and therefore~$\unit 
\oplus\sign(\sigma)\bigodot_{i\in [n]} A_{i,\sigma(i)}=\unit,\ \forall \sigma\in\allperm_{[n]}\setminus\{\Id\}.$  
As a result
$$\unit\oplus \sign(\sigma|_{\cli})\bigodot_{j\in \cli}A_{j,\sigma(j)}\bigodot_{j\in I\setminus\cli}A_{j,j}=
\unit\oplus (\ominus\unit)^{|\cli|-1}\bigodot_{j\in \cli}A_{j,\sigma(j)}=\unit,\ \forall \cli\ne\{i\},\text{ and}$$
 $$\unit \oplus\sign(\sigma)\bigodot_{i\in I} A_{i,\sigma(i)} \bigodot_{i\in I^c} A_{i,i}
=\unit \oplus
\sign(\sigma)\bigodot_{i\in [I]} A_{i,\sigma(i)}=\unit,\ \forall \sigma\in\allperm_{I}\setminus\{\Id\}\enspace,$$
which shows Points~\eqref{dfs1} and \eqref{dfs2}.

Finally, consider $\sigma$ and $\tau$ as in Point~\eqref{dfs3}.
Let~$\pi\in\allperm_{[n]}$ s.t.~$\pi|_{\{t\}^c}=\sigma\in\allperm_{\{t\}^c,\{s\}^c}$ and~$\pi(t)=s$, and denote by~$\bar{\bar{s}}$ the element of $C_\pi$ containing $s$.
From Proposition~\ref{resign}, we have that
$\sign(\sigma)(\ominus\unit)^{t+s}=\sign(\pi)=\sign(\pi|_{\bar{\bar{s}}})\sign(\pi|_{\bar{\bar{s}}^c})=(\ominus\unit)^{|\bar{\bar{s}}|-1}\sign(\pi|_{\bar{\bar{s}}^c}),$
Notice that~$|\bar{\bar{s}}|-1=|\bar{s}|$ and that the path~$
\bigodot_{j\in\bar{s}}A_{j,\sigma(j)}$  from~$s$ to~$t$ is the only elementary open path in~$\sigma$. Using Point~\eqref{dfs1}, we get
\begin{eqnarray*}\sign(\sigma)\bigodot_{\bar{i}\in C_\sigma}\bigodot_{j\in\bar{i}}A_{j,\sigma(j)}&=&
(\ominus\unit)^{|\bar{\bar{s}}|-1}(\ominus\unit)^{t+s}\sign(\pi|_{\bar{\bar{s}}^c})
\bigg(\bigodot_{\substack{\bar{i}\in C_\pi:\\\bar{i}\subseteq \bar{\bar{s}}^c}}
\bigodot_{j\in\bar{i}}A_{j,\pi(j)}\bigg)\bigg(
\bigodot_{j\in\bar{s}}A_{j,\sigma(j)}\bigg)\\&\curlyeqprec&
(\ominus\unit)^{|\bar{s}|}(\ominus\unit)^{s+t}
\bigodot_{j\in\bar{s}}A_{j,\sigma(j)},\end{eqnarray*}
where $(\ominus\unit)^{|\bar{s}|}(\ominus\unit)^{s+t}$ is the sign of the elementary bijection~$\tau\in\allperm_{\{t\}^c,\{s\}^c}$, corresponding to the elementary path~$\bigodot_{j\in\bar{s}}A_{j,\sigma(j)}=\bigodot_{j\in\bar{s}}A_{j,\tau(j)}$ of $\sigma$ (Corollary~\ref{eps}).
\end{proof}

\begin{df}
A \myemph{quasi-identity} (or pseudo-identity) matrix over $\SS$ is a nonsingular, multiplicatively idempotent matrix,  
with $\unit$ on the diagonal, and off-diagonal entries in~$\SS^\circ$.

\end{df}

\begin{df}\label{adj} The~$\mathbf{(r,c)}$-\myemph{submatrix}~$A_{(r,c)}$ of a matrix~$A=(a_{i,j})\in 
\SS^{n\times n}$ is obtained by deleting  row~$r$ and column~$c$ of~$A$, and its determinant is called
 the~$\mathbf{(r,c)}$-\myemph{minor} of $A$.
 The \myemph{adjoint matrix}  of~$A$ is defined as~$\adj(A)=(a'_{i,j} ),$ where $$a'_{i,j} =
(\ominus\unit)^{ i+j}\det(A_{(j,i)})\enspace .$$
  Notice that~$\det(A_{(j,i)})$ is obtained as the sum corresponding to all permutations passing through~$(j,i)$, with~$a_{j,i}$  
 removed  $$\det(A_{(j,i)})=\bigoplus_{\substack{\pi\in \allperm_{[n]}:\\\pi(j)=i}}\sign(\pi|_{\{j\}^c}) 
a_{1,\pi(1)}\cdots a_{j-1,\pi(j-1)} a_{j+1,\pi(j+1)}\cdots a_{n,\pi(n)}.$$ 
Writing  each permutation as the product~of  disjoint cycles, and applying Proposition~\ref{resign} for~$I=\{j\}^c$, we get 
\begin{equation}a'_{i,j}=\bigoplus_{\substack{\pi\in \allperm_{[n]}:\\\pi(j)=i}}\sign(\pi)
\big(a_{i,\pi(i)}a_{\pi(i),\pi^2(i)}\cdots a_{\pi^{-1}(j),j}\big) \bigodot_{\substack{[k]:\\i\notin [k]
}}\bigodot_{j\in [k]}a_{j,\pi(j)}.
\label{formulaadjoint}
\end{equation} 

If~$\det(A)\in \SS^{*}$, we denote~$A^{\nabla}=\adj(A)\det(A)^{ -1},$ and call it \myemph{the quasi-inverse} of~$A$.
We also denote $\mathcal{I}_A:=AA^{\nabla}$ and~$\mathcal{I}'_A:= A^{\nabla}A$.
\end{df}

\begin{pro}\label{qasinv} Let $A\in \SS^{n\times n}$ be a nonsingular matrix.
We have
$$\mathcal{I}_A\succeq^\circ\mathcal{I},\;\mathcal{I}'_A\succeq^\circ\mathcal{I}\;\text{and}\;  
(\mathcal{I}_A)_{i,i}=(\mathcal{I}'_A)_{i,i}=\unit,\; i\in [n]\enspace .$$
Moreover, if~$\SS$ has a modulus, then~$|\det(\mathcal{I}_A)|=|\det(\mathcal{I}'_A)|=\unit,$
and if~$\SS=\supertropical$, then~$\mathcal{I}_A$ and~$\mathcal{I}'_A$
are quasi-identities.
\end{pro}
\begin{proof}
The first assertion can be deduced from~\cite[Proposition~2.1.2]{G.Ths}.
See also~\cite{reutstraub}.
The property that $\mathcal{I}_A$ and $\mathcal{I}'_A$
are quasi-identities when $\SS=\supertropical$
is shown in~\cite[Theorem~2.8]{STMA2}. 
When $\SS$ has a modulus, applying the modulus to the expression of
$\mathcal{I}_A$, we obtain that $|\mathcal{I}_A|=\mathcal{I}_{|A|}$,
hence applying the previous property to $|A|$, we get the 
second assertion of the proposition, that is
$|\det(\mathcal{I}_A)|=|\det(\mathcal{I}'_A)|=\unit$.
\end{proof}

\begin{df} Let~$S=\{s_i:\ i\in[n]\}$ be a totally ordered set of
cardinality $n$, with~$s_i<s_j\ \forall i<j$, and let $k\in\{0,\ldots, n\}$.
The~$k$th \myemph{compound matrix} of a $S\times S$ matrix~$A=(a_{i,j})$ with entries in $\SS$ is  
the~$\mathcal{P}_k(S)\times \mathcal{P}_k(S)$ matrix with entries in $\SS$, denoted~$A^{\wedge k}$, 
defined by $$A^{\wedge k}_{_{I,J}}=\bigoplus_{\sigma\in \allperm_{I,J}}\sign(\sigma)\bigodot_{s\in I}A_{s,\sigma(s)}
,\ \ I,J\in \mathcal{P}_k(S)\enspace ,$$
and identified as a $\left(\substack{n\\k}\right)\times \left(\substack{n\\k}\right)$ matrix. 
Identification,
signs of bijections in~$A^{\wedge k},$ and signs of entries  in~$\adj(A^{\wedge k}),$ are taken with 
respect to  the lexicographic order in~$\mathcal{P}_k(S)$.
\end{df}

\begin{rem}\label{compinv}
\begin{myenumerate}
\item\label{compinv1} In particular if $S=[n]$, then
 $$A^{\wedge k}_{_{I,J}}=\bigoplus_{\sigma\in \allperm_{I,J}}\sign(\sigma)\bigodot_{i\in I}
a_{i,\sigma(i)} ,\ \ I,J\in \mathcal{P}_k([n]),\; k\in\{0,\ldots, n\}\enspace .$$
So 
\begin{equation*}\label{n-1adj}A^{\wedge k}=\begin{cases}\unit&,\ k=0\\A&,\ k=1\end{cases},\ 
\text{ and }\ A^{\wedge n-1}=Q^{-1}\adj(A)^\T Q\in\SS^{n\times n},\end{equation*}
 where~$A^\T$ denotes the transpose of matrix $A$,  and~$Q$ is the monomial matrix defined 
by~$Q_{i,n+1-i}=(\ominus\unit)^i$, and ~$\zero$ otherwise,
which is orthogonal and satisfies $Q^T=(\ominus\unit)^{n+1}Q$.  That is~$\adj(A)=
Q(A^{\wedge n-1})^\T Q^{-1},$ or equivalently,~$\adj(A)_{i,j}=(\ominus\unit)^{i+j}
A^{\wedge n-1}_{\{j\}^c,\{i\}^c}.$
If~$A$ is diagonal, then the transpose and the entry-signs vanish 
so~$\adj(A)=P_\pi A^{\wedge n-1}P_\pi,$  where~$\pi(i)=n+1-i$,~$i\in [n]$. 

\item\label{compinv2}  
For $k\in \{0,\ldots, n\}$ and $\pi\in \allperm_{S}$, 
 define the maps  $g^{k,\pi},f^{k,\pi}:\mathcal{P}_k(S)\rightarrow\SS$ by $$g^{k,\pi}(I)=\sign(\pi|_{I})\ \text{ and }\ 
f^{k,\pi}(I)=\sign(\pi|_{\pi^{-1}[I]}), \; \text{for} \; I\in \mathcal{P}_k(S)\enspace .$$
Note that 
$g^{k,\pi}(I)=\sign(\pi|_{I})=\sign(\pi^{-1}|_{\pi[I]})=f^{k,\pi^{-1}}(I),$  for all $I\in \mathcal{P}_k(S)$.

Recall that for  $\pi\in \allperm_{S}$, 
$\pi^{(k)}\in\allperm_{\mathcal{P}_k(S)}$ is s.t.~$\pi^{(k)}(I)=\pi[I]$
for all $I\in \mathcal{P}_k(S)$.

If~$A$ is the permutation matrix~$P_\pi$ associated to~$\pi\in\allperm_{S}$, then
\[ A^{\wedge k}_{I,J}= \begin{cases}
\sign(\pi|_{I}) &,\ J=\pi[I]\\
\zero&,\ \text{otherwise} 
\end{cases}\]
so that it can be written as $A^{\wedge k}=D_{g^{k,\pi}}P_{\pi^{(k)}}=P_{\pi^{(k)}}D_{f^{k,\pi}}$.

If~$A$ is the  $S\times S$ diagonal matrix~$D_q$ with diagonal entries given by the map~$q:S\rightarrow\SS, \; 
i\mapsto q_i$, then~$A^{\wedge k}$ is the diagonal matrix $D_{h}$ with diagonal entries given by the map $h:
\mathcal{P}_k(S)\rightarrow\SS, \;
I \mapsto h_I=\bigodot_{i\in I}q_i$.

As a result of the (classical) Cauchy-Binet formula, if~$A=D_{q}P_\pi$ (resp.~$P_\pi D_{q}$), then $$A^{\wedge k}
=D_{h}D_{g^{k,\pi}}P_{\pi^{(k)}}
\ \ \text{ (resp. }P_{\pi^{(k)}}
D_{f^{k,\pi}}D_{h}\text{)}.$$
\end{myenumerate}
\end{rem}

\begin{pro} \label{CBF} (Tropical Cauchy-Binet,~\cite[Proposition 2.18]{G.Ths})
If~$A,B\in \SS^{n\times n}$, then $$(A B)^{\wedge k}\succeq^\circ A^{\wedge k} B^{\wedge k}.$$ 
\end{pro}

\begin{thm}\label{ADJ} For $A,B\in \SS^{n\times n}$, we have
$$\adj(A B)\succeq^\circ \adj(B)\adj(A).$$  
\end{thm} 

\begin{proof}
Using the properties in~\eqref{compinv1} of Remark~\ref{compinv}, this theorem follows from Proposition~\ref{CBF}:
\begin{align*}
&\adj(AB)=Q\big((AB)^{\wedge n-1}\big)^{\T}Q^{-1}\succeq^\circ Q(A^{\wedge n-1}B^{\wedge n-1})^{\T}Q^{-1}=\\
&Q(B^{\wedge n-1})^\T(A^{\wedge n-1})^{\T}Q^{-1}
=Q(B^{\wedge n-1})^{\T}Q^{-1}Q(A^{\wedge n-1})^{\T}Q^{-1}=\adj(B)\adj(A)\enspace .\qed \end{align*}
\renewcommand{\qedsymbol}{}
\end{proof}

\begin{cor}\label{eqadj} Let~$\SS$ be naturally ordered, and let~$A,B\in \SS^{n\times n}$.
If~$A$ is invertible, then
\begin{enumerate}
\item\label{eqadj-1}   Equality holds in Theorem~\ref{ADJ}, and~$A^\nabla=A^{ -1}.$

\item\label{eqadj-2} Equality holds in Proposition~\ref{CBF}, and~$(A^{ -1})^{\wedge k}=(A^{\wedge k})^{ -1}$.

\end{enumerate}
Moreover, similar assertions hold if $B$ is invertible.
\end{cor}

\begin{proof}Let~$A$ be invertible. Since~$\SS$ is naturally ordered, it can 
be written~$A=D_{q}P_\pi$. So~$A^{-1}=P_{\pi^{-1}}D_{q^{-1}}$.

\eqref{eqadj-1} 
From~\cite[Lemma 3.6]{AGG14}, equality holds in Theorem~\ref{ADJ}
when $A$ or $B$ is invertible. Therefore, using Theorem~\ref{detAB},
and Definition~\ref{adj}, we get 
\begin{eqnarray*}
A^\nabla=\det(P_\pi)^{-1} \adj(P_\pi)\det(D_{q})^{-1}\adj(D_{q})=P^{-1}_\pi D_{q^{-1}} 
=A^{-1}.\end{eqnarray*}

\eqref{eqadj-2}
From~\eqref{compinv2} of Remark~\ref{compinv} applied to~$A$ and~$A^{-1}$,~$k\in\{0,\ldots, n\}$ being fixed, 
we have~$A^{\wedge k}=D_{h}D_{g^{k,\pi}}P_{\pi^{(k)}}$ and
$$(A^{-1})^{\wedge k}=P_{(\pi^{(k)})^{-1}}D_{f^{k,\pi^{-1}}}D_{h^{-1}},\;\text{with}\; h_I=
\bigodot_{i\in I}q_i\; \text{for}\; I\in\mathcal{P}_k(S)\enspace .$$
Since~$g^{k,\pi}=f^{k,\pi^{-1}}$ and $D_{g^{k,\pi}}$ is involutory, 
we get~$(A^{\wedge k})^{-1}=(A^{-1})^{\wedge k}$.

Now, applying Proposition~\ref{CBF} to $A$ and $B$ and then to 
$A^{-1}$ and $AB$, we get that
$$(A B)^{\wedge k}\succeq^\circ A^{\wedge k} B^{\wedge k}
\succeq^\circ A^{\wedge k}(A^{-1})^{\wedge k} (AB)^{\wedge k}\text{ and }A^{\wedge k}
(A^{-1})^{\wedge k} (AB)^{\wedge k}=
(AB)^{\wedge k}.$$ Since $\succeq^\circ$ is an order, we deduce
the equality.

\end{proof}

Note that when $k=n$, Proposition~\ref{CBF} and its equality 
were already stated in Theorem~\ref{detAB}.

\begin{lem} \label{n-1def}If~$A\in\TT^{n\times n}$ is definite, then~$A^{\wedge k}_{_{I,I}}=\unit,\forall I
\in\mathcal{P}_k([n]), \forall k\in[n].$\end{lem}

\begin{proof} From Point~\eqref{dfs1} of Lemma~\ref{dfs}, 
$$A^{\wedge k}_{_{I,I}}=\unit\oplus\bigoplus_{\sigma\in\allperm_I\setminus\{\Id\}}
\sign(\sigma)\bigodot_{i\in I}A_{i,\sigma(i)}=\unit\enspace .\qed$$
\renewcommand{\qedsymbol}{}
\end{proof}

\begin{rem}
Notice that~$A^{\wedge k}$ is not necessarily definite when~$A$ is definite. 

For example, 
$\left(\begin{array}{ccc}\unit&\unit&\zero\\\zero&\unit&\unit\\
\unit&\zero&\unit\end{array}\right)^{\wedge 2}
=\left(\begin{array}{ccc}\unit&\unit&\unit\\\ominus\unit&\unit&\unit\\
\ominus\unit&\ominus\unit&\unit\end{array}\right)$ over~$\mathbb{SR}_{\max}$, which is singular.
\end{rem}

\begin{df} Let~$A$ be a matrix over~$\TT$ and let $A^0=\mathcal{I}$. If~$\bigoplus_{k\geq 0}A^k$ 
converges to a matrix over~$\TT$, then this matrix is defined as the \myemph{Kleene star} of~$A$,  
denoted by~$A^*$.
\end{df}

\begin{thm}\label{DF}  Let $A=\mathcal{I}\ominus B$  over $\TT$,
with $B_{i,i}=\zero$, $i\in [n]$. Assume that $A$ is definite.
Then $|A|^*$ exists and is definite, the weight of every cycle in $B$
is $\curlyeqprec\ominus \unit$,
 and \begin{equation}\label{**}|B|^*=|A|^*=|A^{ k}|=|\adj(A)|=
|A^\nabla|=|A^{\nabla\nabla}|,\ \forall k\geq n-1\enspace.\end{equation}
If $A^*$ exists, then $|A^*|=|A|^*$.
Moreover,   if $\TT$ is idempotent, then  \begin{equation}\label{***}A^{**}=A^*=
(\ominus B)^*\enspace .\end{equation}
If the weight of every cycle in $B$ is  $\curlyeqprec\unit$
(in particular if $\ominus \unit=\unit$ or if
the modulus of the weight of every cycle in $B$ is  
strictly dominated by $\unit$), $\TT$ being 
not necessarily idempotent, then $B^*$ exists and \begin{equation}\label{starnabla}A^\nabla=B^*.
\end{equation}
\end{thm}

\begin{proof}
Since $A$ is definite and $A_{i,j}=\ominus B_{i,j}$ when $i\neq j$, 
Point~\eqref{dfs2} of Lemma~\ref{dfs} 
implies that the weight of every cycle in $B$
is $\curlyeqprec\ominus \unit$.
Since the modulus is a morphism, we have
$\det(|A|)=|\det(A)|$,  $|A|$ is definite when $A$ is definite, 
$|A^k|=|A|^k$, $|\adj(A)|=\adj(|A|)$,
$|A^\nabla|=(|A|)^\nabla$, $|A^{\nabla\nabla}|=(|A|)^{\nabla\nabla}$,
and $|A^*|=|A|^*$ when both matrices exist.
So to show the two first assertions, it is sufficient to show
that $A^*$ exists and that~\eqref{**} holds if $A$ is a definite
matrix over $\Mod$.
In this case, the existence of $A^*$ and 
the first and second equalities of~\eqref{**} date back to the~60's (see ~\cite{YL}), and 
have received several proofs, using various techniques (see for~instance~\cite[Theorems~2 and~6]{PAFM}
and~\cite[Theorem~3.9]{AGG14}). The third equality in~\eqref{**} is true by definition, since~$A$ is 
definite, and the last equality has been proved in~\cite[Lemma~6.7 and Claim~6.8]{FTM}. The main 
argument is made due to Lemma~\ref{dfs}, by factoring  into cycles. As a result, the diagonal entries of 
the $A^{\nabla\nabla}$ are~$\unit$, and the~$i,j$ off-diagonal entry is  the sum of dominant elementary 
paths from~$i$~to~$j$. 

When $\TT$ is idempotent, then the equalities in~\eqref{***} hold since
$$A^k=(\mathcal{I}\ominus B)^k=\bigoplus_{m=0}^k  (\ominus B)^{m}\ \text{ implies }\ A^*=
\bigoplus_{k\geq 0}A^k=\bigoplus_{k\geq 0}(\mathcal{I}\ominus B)^k=(\ominus B)^*.$$

Assume now that the weight of every cycle in $B$ is  $\curlyeqprec\unit$.
Because it is already $\curlyeqprec\ominus\unit$, this holds when 
$\ominus \unit=\unit$. Also, from Property~\ref{prop1},
this also holds when the modulus of the weight of every cycle in $B$ is  
strictly dominated by $\unit$.
By Point~\eqref{dfs3} of Lemma~\ref{dfs} and~\eqref{formulaadjoint}, 
 we get that the~$i,j$ entry of $A^\nabla$ is the sum
$\bigoplus\sign(\pi)A_{i,\pi(i)}\cdots A_{\pi^{-1}(j),j},$ where~$\pi$ is  a permutation s.t.~$\pi(j)=i$  
with at most one nontrivial cycle. Denote $k_{\pi}$ the length of the nontrivial cycle of $\pi$, with $k_{\pi}=1$
 if there is no such cycle, then $\sign(\pi)=(\ominus\unit)^{k_{\pi}-1}$.
Since $A_{i,j}=\ominus B_{i,j}$ when $i\neq j$, 
this implies that the~$i,j$ entry of $A^\nabla$ is the sum of 
 the weights of all elementary paths from $i$ to $j$ with respect to the 
weight matrix $B$, when $i\neq j$, and that
it is equal to $\unit$ if $i=j$.

By definition, the~$i,j$ entry of $B^*$ is the sum of the weights of all paths from $i$ to $j$ with respect to the 
weight matrix $B$ (when this sum converges).
Since any path  from $i$ to $j$ is the product of an elementary path 
from $i$ to $j$ and of not necessarily disjoint cycles,
and since the weight of every cycle in $B$ is  $\curlyeqprec\unit$, 
we deduce that  the~$i,j$ entry of $B^*$ is equal
to the sum of the weights of elementary paths from $i$ to $j$ when $i\neq j$,
or to $\unit$ when $i=j$, and that $B^*$ exists.
\end{proof}

\begin{thm}[Frobenius property. See~\protect{\cite[Remark~1.3]{STMA}}] \label{FP} 
If $\SS=\Mod$ or $\SS=\supertropical$, then 
 $$(a\oplus b)^{ n}=a^{ n}\oplus b^{ n},\ \forall a,b\in \SS\enspace . $$
\end{thm}


\section{Jacobi's Identity}\label{jacobi}

Over a ring, one can easily obtain the  identity of  Jacobi (see for instance~\cite[Section~1.2]{Fallat&Johnson}) 
$$\det(A)\big(DA^{-1}D\big)^{\wedge n-k}_{J^c,I^c}= A^{\wedge k}_{_{I,J}},\text{ where }D_{i,i}=(-1)^i,\text{ and } 
0\text{ otherwise},$$ from the multiplicativity of the compound matrix  and the multiplicativity of  the determinant function. 
The following result shows that the same holds in rather general  semirings if $A$ is 
invertible.

\begin{lem}\label{IEP} 
Let $\SS$ be naturally ordered and~$A\in\SS^{n\times n}$ be invertible.
Consider the diagonal matrix $D$ over $\SS$ 
such that $D_{i,i}=(\ominus \unit)^i$, $i\in [n]$.
Then for every $I,J\in \mathcal{P}_k([n])$ 
$$\det(A) \big(DA^{-1}D\big)^{\wedge n-k}_{J^c,I^c}=(A)^{\wedge k}_{I,J},\ \ \forall k\in\{0,\ldots, n\}.$$
\end{lem}

\begin{proof}
Since $\SS$ is naturally ordered, an invertible matrix $A$
is necessarily of the form $A=D_{g}P_\pi$ with
$g_i\in \SS^*,\ i\in[n]$, and~$\pi\in\allperm_{[n]}$.
Moreover, using Point~\eqref{eqadj-2} of Corollary~\ref{eqadj}
(with $k$ and $n$) and $D^{-1}=D$, we
only need to prove the equality in the lemma when $A$ is the diagonal matrix $A=D_{g}$ and when $A$ is the
 permutation matrix $A=P_\pi$.

The lemma holds for an invertible diagonal matrix $A=D_{g}$, since,
when $I=J$, 
$$\det(A) \big(DA^{-1}D\big)^{\wedge n-k}_{J^c,I^c}= \bigg(
\bigodot _{i\in [n]}g_i \bigg) \big(D_g^{-1}\big)^{\wedge n-k}_{J^c,I^c}
=\bigg(\bigodot _{i\in [n]}g_i  \bigg)\bigg(\bigodot _{i\in I^c}g_i^{-1}\bigg)
=\bigodot_{i\in I}g_i=A^{\wedge k}_{_{I,J}},$$
and both sides of the equality are equal to $\zero$ otherwise. 

When $A=P_\pi$, then $A^{-1}=P_{\pi^{-1}}$. Using Point~\eqref{compinv2} of
Remark~\ref{compinv}, and then Proposition~\ref{resign}, 
we get
$$\det(A) \big(DA^{ -1}D\big)^{\wedge n-k}_{J^c,I^c}= \begin{cases}\sign(\pi) 
\bigodot _{_{j\in J^c}}(\ominus\unit)^{j} 
\sign(\pi^{-1}|_{J^c}) \bigodot _{_{i\in I^c}}(\ominus\unit)^{ i}&\text{if}\; \pi^{-1}[J^c]=
I^c\\ \zero&\text{otherwise}
\end{cases}$$

$$=\begin{cases}\sign(\pi|_{I})&\text{if}\;\pi[I^c]=J^c\\\zero&\text{otherwise}\end{cases}=
\begin{cases}
\sign(\pi|_{I})&\text{if}\; \pi[I]=J\\\zero&\text{otherwise}\end{cases}=A^{\wedge k}_{_{I,J}}\enspace .\qed $$
\renewcommand{\qedsymbol}{}
\end{proof}

The following theorem uses Corollary~\ref{eqadj} to generalize Lemma~\ref{IEP} to any 
nonsingular matrix, using its definite form.

\begin{thm}[Tropical Jacobi]\label{QDF}
 If~$A\in\TT^{n\times n}$ is nonsingular, then~$\forall I,J\in \mathcal{P}_k([n])$  
$$\det(A)\big(DA^\nabla D\big)^{\wedge n-k}_{J^c,I^c}=\det(A)\big(\det(A)^{-1}
A^{\wedge n-1}\big)^{\wedge n-k}_{\pi(I),\pi(J)}\succeq^\circ A^{\wedge k}_{_{I,J}},\ \
 \forall k\in\{0,\ldots ,n\},$$
where $\pi:\mathcal{P}_k([n])\rightarrow \mathcal{P}_{n-k}(\mathcal{P}_{n-1}([n]))$ is defined 
by $I\mapsto \{\{i\}^c:\ i\in I^c\}.$\end{thm}

\begin{proof}  
From Point~\eqref{compinv1} of Remark~\ref{compinv}, we have for all nonsingular matrices
$A$, $DA^\nabla D=\det(A)^{-1}DQ(A^{\wedge n-1})^\T Q^{-1}D$, with $Q$ as in 
Remark~\ref{compinv}. Then $DQ=P_\sigma$ and $Q^{-1}D=P_{\sigma}^{-1}$ where
$\sigma(i)=n+1-i$ or simply $\sigma(i)=\{i\}^c$ when $A^{\wedge n-1}$ is indexed
by the sets $\{i\}^c$. Since the bijection $\pi$ of the theorem satisfies
$\pi(I^c)=\sigma^{(n-k)}(I)$, we get the (first) equality of the theorem.

From Point~\eqref{eqadj-1}  of Corollary~\ref{eqadj}, Lemma~\ref{IEP} and the first 
equality of the theorem, the (second) inequality  of the theorem is true (and is an equality) for all invertible 
matrices. Then, using any definite form of a nonsingular matrix, $A=P\bar{A}$, and using 
Corollary~\ref{eqadj} and $D=D^{-1}$,
we see that it is sufficient to prove the inequality when~$A$ is definite, in which case it reduces to
\begin{equation}\label{ineq-when-definite}
\big(A^{\wedge n-1}\big)^{\wedge n-k}_{\pi(I),\pi(J)}\succeq^\circ A^{\wedge k}_{_{I,J}},\ \
 \forall k\in\{0,\ldots ,n\}\enspace .
\end{equation}
We thus assume now that $A=(a_{i,j})$ is definite.
Then, from Point~\eqref{dfs3} of Lemma~\ref{dfs}, we get that $A^{\wedge n-1}_{\{i\}^c,\{j\}^c}$ is the 
sum of signed elementary bijections, with the elementary path being from~$j$ to~$i$.
In particular, it is equal to $\unit$ when $i=j$.
Then, $\big(A^{\wedge n-1}\big)^{\wedge n-k}_{\pi(I),\pi(J)}$ is the sum of signed 
products of signed elementary bijections. 

In order to prove Inequality~\eqref{ineq-when-definite}, we shall show 
the following properties~:
\begin{enumerate}
\item\label{incl1} every signed bijection in $A^{\wedge k}_{_{I,J}}$ is a \summand in$\big(A^{\wedge n-1}\big)^{\wedge n-k}_{\pi(I),\pi(J)}$, 
\item\label{incl2}  every other \summand in~$\big(A^{\wedge n-1}\big)^{\wedge n-k}_{\pi(I),\pi(J)}$ 
reappears, with an opposite sign, creating an element of $\TT^\circ$
which is added to  $A^{\wedge k}_{_{I,J}}$.
\end{enumerate}
for every~$I, J\in\mathcal{P}_k([n])$, up to terms that are  $\curlyeqprec$ to other terms in the same sum (and then can be omitted).
Property~\eqref{incl2} means that an element in~$\mathcal{T}^\circ$
is added to  $A^{\wedge k}_{_{I,J}}$. Moreover,
due to Property~\ref{propnew} and Lemma~\ref{lem-propnew}, we do not need 
to count the number of times a term appear in~\eqref{incl2},
so that~\eqref{incl1} and~\eqref{incl2} are sufficient to 
prove~\eqref{ineq-when-definite}.

Proof of~\eqref{incl1}:
  If~$I= J$, from Lemma~\ref{n-1def}, $A^{\wedge k}_{_{I,I}}=
\unit,\ \forall I\in \mathcal{P}_k([n]),\forall k\in[n]$,
which  corresponds to the identity permutation in
 $\big(A^{\wedge n-1}\big)^{\wedge n-k}_{\pi(I),\pi(I)}$. 

Assume now that $I\neq J$.
   Using~\eqref{cycfact}, we factor every bijection in~$A^{\wedge k}_{_{I,J}}$ into disjoint 
 elementary open paths and  cycles
\begin{eqnarray}\label{tau}
\lefteqn{\sign(\tau)\bigodot_{\cli\in C_\tau}\bigodot_{j\in \cli}a_{j,\tau(j)}}
\\ \nonumber
& =& \sign(\tau)\bigodot_{\substack{\cli\in C_\tau:\\\cli\subseteq J}}\bigodot_{j\in \cli}a_{j,\tau(j)}
\bigodot_{j\in I\setminus J} a_{j,\tau(j)} a_{\tau(j),\tau^2(j)}\cdots 
a_{\tau^{m_j-1}(j),\tau^{m_j}(j)}\enspace,
\end{eqnarray}
where $\tau\in\allperm_{I,J}$, and 
$\tau^{m_j}(j)\in J\setminus I$ for all $j\in I \setminus J$.
Define $\sigma\in\allperm_{[n]}$ s.t.~$\sigma|_I=\tau$, $\sigma(\tau^{m_j}(j))=j\ \forall j\in I\setminus J$, and $\sigma|_{(I\cup J)^c}=\Id$ and let 
$K= \bigcup_{\cli\in C_\tau:\;\cli\subseteq J} \cli\subset I\cap J$. 
Using Proposition~\ref{resign}, we obtain 
\begin{align*}
\underbrace{\sign(\sigma|_{K})}_{\text{cycles of $\tau$}}\underbrace{\sign(\sigma|_{K^c})}_{\substack{\text{cycles induced from}\\\text{open paths of $\tau$}}}&=\sign(\sigma)\\
&=\sign(
\sigma|_I
)\sign(\sigma|_{I^c})(\ominus\unit)^{\sum_{i\in I} i+\sigma(i)}\\
&=\sign(\tau)\sign(\sigma|_{I^c})\bigodot_{j\in I\setminus J}(\ominus\unit)^{j+\sigma^{m_j}(j)}.
\end{align*}
From Point~\eqref{dfs3} of Lemma~\ref{dfs}, $\sign(\sigma|_{K})\bigodot_{
\cli\in C_\tau:\;\cli\subseteq J}\bigodot_{j\in \cli}a_{j,\tau(j)}\curlyeqprec\unit,$
 therefore~\eqref{tau} is either equal or $\curlyeqprec$ to
the following  bijection in which the cycles in $K$ are replaced by loops:
\begin{equation}\label{tau2}
\sign(\sigma|_{K^c})\sign(\sigma|_{I^c})\bigodot_{j\in I\setminus J}(\ominus\unit)^{j+\sigma^{m_j}(j)}
\bigodot_{j\in I\setminus J} a_{j,\tau(j)} a_{\tau(j),\tau^2(j)}\cdots 
a_{\tau^{m_j-1}(j),\tau^{m_j}(j)}\enspace.\end{equation}
For $j\in I\setminus J$, we take~$\sigma^{(j)}=(j\ \tau(j)\cdots \tau^{m_j}(j))\in\allperm_{[n]}$.
Note that $\sigma|_{K^c}$ is obtained as the composition of the $\sigma^{(j)},\ j\in I\setminus J$, hence $\sign(\sigma|_{K^c})=\bigodot_{j\in I\setminus J}\sign(\sigma^{(j)})$.
Moreover, using Proposition~\ref{resign} again, we get
$$\sign(\sigma^{(j)})\sign(\sigma^{(j)}|_{  \{  \sigma^{m_j}(j)  \}^c  })=(\ominus\unit)^{\sigma^{m_j}(j)+j}.$$
Therefore,~\eqref{tau2} is reduced to
\begin{equation}\label{1cont}\sign(\sigma|_{I^c})\bigodot_{j\in I\setminus J}
\underbrace{\sign(\sigma^{(j)}|_{\{\sigma^{m_j}(j)\}^c}) a_{j,\sigma(j)} a_{\sigma(j),
\sigma^2(j)}\cdots a_{\sigma^{m_j-1}(j),\sigma^{m_j}(j)}},\end{equation}
 where the  underbraced  expression is a signed elementary bijection in $A^{\wedge n-1}_{\{\sigma^{m_j}(j)\}^c,\{j\}^c}$ 
corresponding to the permutation
$\sigma^{(j)}$.

Since $\{j\}^c\in\pi(J)\setminus\pi(I)\ \Leftrightarrow\ j\in I
\setminus J \ \Leftrightarrow\ \tau^{m_j}(j)\in J\setminus I\ \Leftrightarrow\ 
\{\tau^{m_j}(j)\}^c\in\pi(I)\setminus\pi(J)$, the map 
$\pi(I)\setminus \pi(J)\rightarrow 
\pi(J)\setminus \pi(I)$, defined by $\{\sigma^{m_j}(j)\}^c\mapsto \{j\}^c$,
is a bijection, which is conjugate to the bijection $\sigma|_{ J\setminus I}$.
Moreover, by taking~$\rho|_{\pi(I)
\cap \pi(J)}=\Id$, it can be extended to a 
bijection~$\rho:\pi(I)\rightarrow \pi(J)$. 
Then,  $\sign(\rho|_{\pi(I)\setminus \pi(J)})=\sign(\sigma|_{J\setminus I})$
and since~$\sigma$ is the identity on $(I\cup J)^c$,
we get that $\sign(\rho)=\sign(\sigma|_{I^c})$.
Therefore, recalling that $A^{\wedge n-1}_{\{i\}^c,\{i\}^c}=\unit$ for all $i\in[n]$,
we get that ~\eqref{1cont} is a \summand in~$\big(A^{\wedge n-1}\big)^{\wedge n-k}_{\pi(I),\pi(J)}$.

Proof of~\eqref{incl2}:  
Consider the expression
\begin{equation}\label{sgnpt}\sign(\rho) \bigodot_{\{i\}^c\in \pi(I)}A^{\wedge n-1}_{\{i\}^c,\{q(i)\}^c} 
\enspace ,\end{equation} 
where~$\rho\in\allperm_{\pi(I),\pi(J)}$ and~$q:I^c\to J^c$ is s.t.~$\{q(i)\}^c=\rho(\{i\}^c)$.
Up to terms that are~$\curlyeqprec$ to other ones,
a \summand in~$\big(A^{\wedge n-1}\big)^{\wedge n-k}_{\pi(I),\pi(J)}$
is obtained by replacing~$A^{\wedge n-1}_{\{i\}^c,\{q(i)\}^c}$ in~\eqref{sgnpt},
  by 
a signed elementary bijection, with the single elementary path being from~$q(i)$ to~$i$.
In particular, when~$q(i)=i$, the latter is replaced by~$\unit$.
In view of the arguments above, 
 such a \summand in~$\big(A^{\wedge n-1}\big)^{\wedge n-k}_{\pi(I),\pi(J)}$
is also a bijection of~$A^{\wedge k}_{_{I,J}}$ if~$\rho$ and the signed
elementary bijections satisfy all of the following conditions:
\begin{enumerate}[label=(\alph*),ref=(\alph*)]
\item \label{condbij1}
for every $ i\in I^c, \; i\ne q(i)$, we have $\{i\}^c\in \pi(I)\setminus \pi(J)$ and $\{q(i)\}^c\in \pi(J)\setminus\pi(I)$ 
(which means~$i\in J\setminus I$ and $q(i)\in I\setminus J$), 

\item \label{condbij2}
for every $ i\in I^c, \; i\ne q(i)$, the intermediate indices of the signed elementary path of $A$ from $q(i)$ to $i$ 
(as a bijection in $A^{\wedge n-1}_{\{i\}^c,\{q(i)\}^c}$) are in $I\cap J$, 

\item \label{condbij3}
the sets of intermediate indices of the elementary paths in the nontrivial signed elementary bijections of $A$ in~\eqref{sgnpt} are disjoint.
\end{enumerate}

Note that these conditions are  satisfied in particular when $\rho$ is the
identity permutation and thus $I=J$.
We need to show that if a \summand 
in~$\big(A^{\wedge n-1}\big)^{\wedge n-k}_{\pi(I),\pi(J)}$
does not satisfy all of the above properties, then it
reappears with an opposite sign.

 \myemph{Condition~\ref{condbij1} fails.}~Then, there exists $i\in I^c, \; i\ne q(i)$,  s.t.~$\{i\}^c\in \pi(I)\cap \pi(J)$ 
(resp.~$\{q(i)\}^c\in \pi(I)\cap \pi(J)$),
which means  that~$i\in I^c\cap J^c$ and $\exists  j\ne q(j)=i\text{ s.t.~}\{q(j)\}^c\in \pi(I)\cap \pi(J)$ 
(resp.~$q(i)\in I^c\cap J^c$ and $\exists k\ne q^{-1}(k)=q(i)\text{ s.t.~}\{q^{-1}(k)\}^c\in \pi(I)\cap \pi(J)$).
W.l.o.g., we consider the first situation. 
Let $b$ be the product of the signed  elementary 
bijections in~$A^{\wedge n-1}_{\{i\}^c,\{q(i)\}^c}$ 
and~$A^{\wedge n-1}_{\{j\}^c,\{q(j)\}^c}$ respectively, and let
$p$ be the corresponding path, that is 
$p$ is the concatenation of the  elementary paths from~$q(i)$ to~$i$
and  from~$q(j)=i$ to~$j$ corresponding to the former
signed elementary bijections.
 If the path $p$ is elementary, then  Corollary~\ref{eps} implies that
 $b$ is a signed elementary bijection (or a  cycle
if $q(i)=j$) with path $p$. Then, $b$ 
appears in the factor~$A^{\wedge n-1}_{\{j\}^c,\{q(i)\}^c}A^{\wedge n-1}_{\{i\}^c,\{i\}^c}$
of the opposite sign bijection~$q\circ (i\;  j)$,
where $(i\; j)$ denotes the transposition of $i$ and $j$.
Then, the \summand considered initially reappears 
in~$\big(A^{\wedge n-1}\big)^{\wedge n-k}_{\pi(I),\pi(J)}$ with an opposite sign.

If the  path~$p$ is not elementary, it may be decomposed into an elementary path $p'$ from~$q(i)$ to~$j$ (or a loop if $q(i)=j$),
and the union of not necessarily disjoint cycles. 
Let $b'$ be the signed elementary bijection of $A^{\wedge n-1}_{\{j\}^c,\{q(i)\}^c}$
corresponding to $p'$ and $a$ be the product of the signed cycles 
corresponding to the cycles in $p$.
Then, $b=ab'$ or $b=\ominus a b'$.
Again, the signed elementary bijection $b'$
appears in the factor~$A^{\wedge n-1}_{\{j\}^c,\{q(i)\}^c}A^{\wedge n-1}_{\{i\}^c,\{i\}^c}$
of the opposite sign bijection~$q\circ (i\;  j)$.
So, the \summand of~$\big(A^{\wedge n-1}\big)^{\wedge n-k}_{\pi(I),\pi(J)}$
considered initially is equal to the product of
a term in~$\big(A^{\wedge n-1}\big)^{\wedge n-k}_{\pi(I),\pi(J)}$
corresponding to the bijection $q\circ (i\;  j)$, with a factor $a'$,
equal either to $a$ or to $\ominus a$.
From Lemma~\ref{dfs}, we have $a\curlyeqprec\unit$.
This implies that  $a'\curlyeqprec\unit$ or $a'=\ominus\unit$.
Indeed, if $a'=a$, then $a'\curlyeqprec\unit$, whereas
if $a'=\ominus a$, we get that $a'\curlyeqprec\ominus\unit$,
and Corollary~\ref{cormoins} shows that 
$a'\curlyeqprec\unit$ or $a'=\ominus\unit$.
Then, if $a'\curlyeqprec\unit$, the \summand considered initially is 
$\curlyeqprec$ to another one in~$\big(A^{\wedge n-1}\big)^{\wedge n-k}_{\pi(I),\pi(J)}$,
whereas if $a'=\ominus\unit$, it is equal to the opposite of another
term.

 \myemph{Condition~\ref{condbij1} holds, but condition~\ref{condbij2} fails.}~We consider an intermediate index~$i\in I^c$ (resp.~$i\in J^c$) 
 in the elementary path of the signed elementary bijection~$\tau$ in~$A^{\wedge n-1}_{\{j\}^c,\{q(j)\}^c}$.
Since~$\{i\}^c\in \pi(I)$ (resp.~$\{i\}^c\in \pi(J)$) and~$\rho\in\allperm_{\pi(I),\pi(J)}$, this index also appears as
 the last  (resp.~first) index of the elementary paths of the signed elementary bijections in~$A^{\wedge n-1}_{\{i\}^c,\{q(i)\}^c}$ 
(resp. $A^{\wedge n-1}_{\{q^{-1}(i)\}^c,\{i\}^c}$).
We shall only consider the first situation, since the second one can be handled 
similarly.
From Remark~\ref{maxpath}, an elementary path can be factored into two non-maximal elementary paths, each of 
which can be extended into an elementary path at its ends, and from Corollary~\ref{eps}, 
 the composition of the corresponding signed  elementary bijections is a signed  elementary bijection.
Let~$\sigma,\sigma_1,\sigma_2\in\allperm_{[n]}$ be such that~$\sigma|_{\{j\}^c}=
\tau,\ \sigma(j)=q(j),\ \sigma_1=(i\ \tau(i)\ \cdots\ \tau^{-1}(j)\ j)$
 and $\sigma_2=(q(j)\ \tau(q(j))\ \cdots\ \tau^{-1}(i)\ i)$.
From Corollary~\ref{eps}, the product of~$A^{\wedge n-1}_{\{i\}^c,\{q(i)\}^c}$ with  the signed elementary bijection~$\tau$  
in~$A^{\wedge n-1}_{\{j\}^c,\{q(j)\}^c}$ satisfies
\begin{equation}\label{compB}A^{\wedge n-1}_{\{i\}^c,\{q(i)\}^c}\ \sign(\tau)(a_{q(j),\tau(q(j))} \cdots 
 a_{\tau^{-1}(i),i} a_{i,\tau(i)} \cdots  a_{\tau^{-1}(j),j})=\end{equation}
$$\underbrace{A^{\wedge n-1}_{\{i\}^c,\{q(i)\}^c}\  \sign(\sigma_1|_{\{j\}^c})(a_{i,\sigma(i)} \cdots  a_{\sigma^{-1}(j),j})}_{(*)}\ \ 
\underbrace{\sign(\sigma_2|_{\{i\}^c})(a_{q(j),\sigma(q(j))} \cdots  a_{\sigma^{-1}(i),i}}_{\text{a signed elementary bijection in }
A^{\wedge n-1}_{\{i\}^c,\{q(j)\}^c}}).$$
Let $b$ be obtained by replacing $A^{\wedge n-1}_{\{i\}^c,\{q(i)\}^c}$
by a signed elementary bijection in $(*)$, and let $p$ be the
corresponding path, that is $p$ is the concatenation of
the path from $q(i)$ to $i$ corresponding to the
signed elementary bijection in $A^{\wedge n-1}_{\{i\}^c,\{q(i)\}^c}$,
and the path from $i$ to $j$ corresponding to $\sigma_1$.
If $p$ is elementary, then, from Corollary~\ref{eps} again,
$b$ is a signed elementary bijection in~$A^{\wedge n-1}_{\{j\}^c,\{q(i)\}^c}$,
then the product of the signed elementary bijections in the factor
$A^{\wedge n-1}_{\{i\}^c,\{q(i)\}^c}A^{\wedge n-1}_{\{j\}^c,\{q(j)\}^c}$ of $q$
reappears in the factor $A^{\wedge n-1}_{\{j\}^c,\{q(i)\}^c}A^{\wedge n-1}_{\{i\}^c,\{q(j)\}^c}$
of the  opposite sign bijection~$q\circ(i\ j)$. 
So as above, the \summand considered initially reappears 
in~$\big(A^{\wedge n-1}\big)^{\wedge n-k}_{\pi(I),\pi(J)}$ with an opposite sign.

If $p$ is not elementary, then it may be decomposed into an elementary path 
$p'$ from~$q(i)$ to~$j$ (note that 
we cannot have a loop due to Condition~\ref{condbij1}),
and the union of not necessarily disjoint cycles. 
Let $b'$ be the signed elementary bijection of $A^{\wedge n-1}_{\{j\}^c,\{q(i)\}^c}$
corresponding to $p'$ and $a$ be the product of the signed cycles 
corresponding to the cycles in $p$.
Then, $b=ab'$ or $b=\ominus a b'$, so
the product of the signed elementary bijections in the factor
$A^{\wedge n-1}_{\{i\}^c,\{q(i)\}^c}A^{\wedge n-1}_{\{j\}^c,\{q(j)\}^c}$ of $q$
is equal to the product of signed elementary bijections in 
the factor $A^{\wedge n-1}_{\{j\}^c,\{q(i)\}^c}A^{\wedge n-1}_{\{i\}^c,\{q(j)\}^c}$
of the  opposite sign bijection~$q\circ(i\ j)$,
times a factor equal to $a$ or $\ominus a$.
With the same arguments as in the first case (in which 
Condition~\ref{condbij1} fails),
we obtain that the \summand considered initially is either
$\curlyeqprec$ to another one in~$\big(A^{\wedge n-1}\big)^{\wedge n-k}_{\pi(I),\pi(J)}$,
or equal to the opposite of such a term.

 \myemph{Conditions~\ref{condbij1},\ref{condbij2} hold, but condition~\ref{condbij3} fails.} We consider~$i\in I\cap J$, such that $i$ is 
an intermediate index in the signed elementary bijections~$\tau$ in~$A^{\wedge n-1}_{\{j\}^c,\{q(j)\}^c}$
 and~$\psi$ in~$A^{\wedge n-1}_{\{t\}^c,\{q(t)\}^c}$. 
Similarly to the previous case,  elementary paths can be factored into  non-maximal   elementary paths, which can be
 extended into   elementary paths, where the composition of the corresponding signed elementary bijections 
is a signed elementary bijection. Let~$\sigma,\phi\in\allperm_{[n]}$ be such that~$\sigma|_{\{j\}^c}=\tau,
\ \sigma(j)=q(j),\ \phi|_{\{t\}^c}=\psi$ and~$\phi(t)=q(t)$.
From Corollary~\ref{eps}, we have 
$\tau=\sigma|_{\{j\}^c}=\sigma_2|_{\{i\}^c}\circ\sigma_1|_{\{j\}^c},$  
and~$\psi=\phi|_{\{t\}^c}=\phi_2|_{\{i\}^c}\circ\phi_1|_{\{t\}^c},$ where
\begin{align*}
&\sigma_1=(i\ \tau(i)\cdots
\tau^{-1}(j)\ j)\in\allperm_{[n]}\ ,\ \sigma_2=(q(j)\ \tau(q(j))\cdots\tau^{-1}(i)\ i)\in\allperm_{[n]},\\
&\phi_1=(i\ \psi(i)\cdots
\psi^{-1}(t)\ t)\in\allperm_{[n]}\ ,\ \phi_2=(q(t)\ \psi(q(t))\cdots\psi^{-1}(i)\ i)\in\allperm_{[n]}.\end{align*} 
As a result, the product $b$ of the corresponding signed elementary bijections
can be factored as follows~:
\begin{eqnarray}\label{compC}
\lefteqn{b=\sign(\tau)(a_{q(j),\tau(q(j))} \cdots  a_{\tau^{-1}(i),i}\ \  a_{i,\tau(i)} \cdots  a_{\tau^{-1}(j),j})}\\
\nonumber 
&&\odot \sign(\psi)(a_{q(t),\psi(q(t))} \cdots  a_{\psi^{-1}(i),i}\  \ a_{i,\psi(i)} \cdots  a_{\psi^{-1}(t),t})\\
&=& \sign(\sigma_2|_{\{i\}^c})(a_{q(j),\sigma(q(j))} \cdots \sign(\phi_1|_{\{t\}^c}) 
  a_{\sigma^{-1}(i),i}) (a_{i,\phi(i)} \cdots  a_{\phi^{-1}(t),t})\label{ptjt}\\
&& \odot \sign(\phi_2|_{\{i\}^c}) (a_{q(t),\phi(q(t))} \cdots  
a_{\phi^{-1}(i),i})\sign(\sigma_1|_{\{j\}^c})(a_{i,\sigma(i)} \cdots  a_{\sigma^{-1}(j),j}) \enspace.\label{pttj}\end{eqnarray}
Let $b_1$ and $b_2$ be equal to~\eqref{ptjt} and~\eqref{pttj} respectively,
and let $p_1$ and $p_2$ be the corresponding paths.
If $p_1$ and $p_2$ are elementary paths, then $b_1$ and $b_2$ are 
signed elementary bijections in~$A^{\wedge n-1}_{\{t\}^c,\{q(j)\}^c}$
and~$A^{\wedge n-1}_{\{j\}^c,\{q(t)\}^c}$ respectively.
So their product $b$ appears
 in the factor~$A^{\wedge n-1}_{\{t\}^c,\{q(j)\}^c}A^{\wedge n-1}_{\{j\}^c,\{q(t)\}^c}$,
of the  opposite sign bijection $\rho\circ(j\ t)$.

If $p_1$ or $p_2$ (or both) is not elementary, then they may be decomposed into   elementary paths  $p'_1$ and $p'_2$ from~$q(j)$ to~$t$, and from~$q(t)$ to~$j$ respectively, and the union of not necessarily disjoint cycles.
Let $b'_1$ and $b'_2$ be signed elementary bijections in~$A^{\wedge n-1}_{\{t\}^c,\{q(j)\}^c}$ and~$A^{\wedge n-1}_{\{j\}^c,\{q(t)\}^c}$, corresponding to the
paths $p'_1$ and $p'_2$,  respectively, and $a$ be the product of 
all the signed cycles of $p_1$ and $p_2$.
Then, $b=b'_1b'_2 a$ or $b=\ominus b'_1b'_2 a$.
We conclude as for the previous cases.
\end{proof}

\begin{cor}\label{QEP} 
Let~$A\in\TT^{n\times n}$ be nonsingular, then
$$\det(A) \tr\big((A^\nabla)^{\wedge n-k}\big)\succeq^\circ \tr\big(A^{\wedge k}\big)\enspace .$$
If $\ominus \unit=\unit$, we have
$$\det(A) \big(A^\nabla\big)^{\wedge n-k}_{J^c,I^c}\succeq^\circ A^{\wedge k}_{_{I,J}}\enspace ,$$
in particular  over~$\supertropical$ we have
$$\det(A) \big(A^\nabla\big)^{\wedge n-k}_{J^c,I^c}\models A^{\wedge k}_{_{I,J}}.$$
\end{cor}
\begin{proof}
By Corollary~\ref{eqadj}, we have $( DA^\nabla D)^{\wedge n-k}=
D^{\wedge n-k}(A^\nabla )^{\wedge n-k}D^{\wedge n-k}$,
and since $D^{\wedge n-k}$ is diagonal and equal to its inverse,
we get that $\big((DA^\nabla D)^{\wedge n-k}\big)_{I,I}=\big((A^\nabla )^{\wedge n-k}\big)_{I,I}$ for all $I\in \mathcal{P}_{n-k}$.
Using Theorem~\ref{QDF}, we deduce the first assertion of the corollary.
The second one follows from the
property that $DA^\nabla D=A^\nabla$ when  $\ominus \unit=\unit$.
\end{proof}

In Section~\ref{tcp} we shall apply Corollary~\ref{QEP} to relate the  characteristic polynomials of a 
matrix and its quasi-inverse, and deduce an analogue relation between their eigenvalues.

\begin{cor} \label{cor1} 
Let~$A\in\TT^{n\times n}$ be nonsingular and assume that $\TT\ne\TT^\circ$
and that all entries of $\big(A^{\wedge n-1}\big)^{\wedge n-k}$ belong to $\TT^\vee$.
Then,  the inequalities of  Theorem~\ref{QDF} and
Corollary~\ref{QEP} become equalities.
\end{cor}
\begin{proof}
This follows from Proposition~\ref{prop2}. 
\end{proof}

\begin{cor}  \label{cor2}
Let~$A\in\TT^{n\times n}$ be nonsingular and assume that $A^*$ exists and
$A=A^*$. Then, $|A^{\wedge n-k}|_{J^c,I^c}= |A^{\wedge k}|_{I,J}$,
for all $k\in \{0,\ldots, n\}$, and $I,J\in \mathcal{P}_k([n])$. \end{cor}

\begin{proof}
Since $A$ is nonsingular and $A=A^*$, we deduce that~$A$ is definite.
From Theorem~\ref{DF}, we have  $|A^\nabla|=|A^*|$ so $|A^\nabla|=|A|$.
Since the modulus is a morphism, we get, for $k\in [n]$,
\[|A^{\wedge n-k}|=  \big(|A |\big)^{\wedge n-k}= \big(|A^\nabla |\big)^{\wedge n-k}=\big(|DA^\nabla D|\big)^{\wedge n-k}=|\big(DA^\nabla D\big)^{\wedge n-k}|\enspace , \]
where all operations on moduli of matrices are done with
respect to the semiring $\Mod$.
Hence, applying the modulus to the inequality of
Theorem~\ref{QDF}, we get $ |A^{\wedge n-k}_{J^c,I^c}|\succeq |A^{\wedge k}_{_{I,J}}|$,
for all  $I,J\in \mathcal{P}_k([n])$.
Applying the same inequality to $n-k$ instead of $k$, we get 
 $ |A^{\wedge k}_{I,J}|\succeq |A^{\wedge n-k}_{J^c,I^c}|$,
thus the equality.
\end{proof}


\section{Other identities on compound matrices}\label{oid}


\subsection{The quasi-inverse matrix}\label{epqi}

We define  the  relation $\succmod$ by:  $a\succmod b\ \Leftrightarrow a\succeq^\circ b\text{ and~}|a|=|b|$. 
With the same arguments as for Theorem~\ref{QDF}, we obtain the following
identities.

\begin{pro}\label{defbar} Let $A\in \TT^{n\times n}$. Recall that $\II_A=A A^\nabla$.

\begin{enumerate}
\item\label{cor341} If~$A$ is definite, then
${A}^\nabla_{i,i}=(\II_A)_{i,i}=
\unit,\ \forall i,$ and 
$$\det(\II_A)\succmod\det(A^\nabla)=\det(A^{\wedge n-1})\succmod\unit\enspace.$$

\item\label{cor342} If $A$ is  nonsingular, then 
\[ 
\det(B)\succmod
\begin{cases}
\det(A)^{ n-1}&,\ B=\adj(A)\\
\det(A)^{ n} &,\ B=A \adj(A)\\
\det(A)^{ -1}&,\ B=A^\nabla\enspace. \end{cases}\]

\end{enumerate}
\end{pro}

\begin{proof}\eqref{cor341} Let $A$ be definite.
From Lemma~\ref{n-1def}, we get~$A^\nabla_{i,i}=A^{\wedge n-1}_{\{i\}^c,\{i\}^c}=\unit $, and from Proposition~\ref{qasinv}, we have  $(\II_A)_{i,i}=\unit$,
for every~$i\in[n]$.
From Remark~\ref{compinv},  $\det(A^\nabla)=
\det(Q(A^{\wedge n-1})^\T Q^{-1})=\det(A^{\wedge n-1})$.

Applying Theorem~\ref{QDF} for~$k=0$, we get~$\det(A^{\wedge n-1})\succeq^\circ\unit$. 
Now by the same arguments as in the proof of Theorem~\ref{QDF},
$\det(A^{\wedge n-1})$ is the sum of signed 
products  of  signed elementary bijections,
each of them being equal to the product of (not necessarily disjoint)
signed cycles possibly times $\ominus \unit$.
Then, each term is equal to $\ominus \unit$ or $\curlyeqprec \unit$. 
Since the term corresponding to the identity is equal to $\unit$,
we obtain that $\det(A^{\wedge n-1})$ equals $\unit$ or $\unit^\circ$, which implies
$|\det(A^{\wedge n-1})|=\unit$ and so $\det(A^{\wedge n-1})\succmod \unit$.

Straightforward,~$\det(AA^\nabla)\succeq^\circ\det(A)\det(A^\nabla)=\det(A^\nabla)
= \det(A^{\wedge n-1})\succmod\unit$. 
As above, one can write $\det(AA^\nabla)$ as the sum of signed 
products  of  signed elementary bijections,
each of them being equal to the product of (not necessarily disjoint)
signed cycles possibly times $\ominus \unit$.
So we obtain again that $\det(AA^\nabla)$  equals $\unit$ or $\unit^\circ$,
which implies $\det(AA^\nabla)\succmod\det(A^\nabla)$.

\eqref{cor342} Denote the right normalization of~$A$ by~$A=\bar{A} P$. From Corollary~\ref{eqadj},
\begin{eqnarray*}
\det(\adj(A))&=&\det(\adj(P)\adj(\bar{A}))=\det(\adj(P))\det(\adj(\bar{A}))\\
&=& (\det(P))^n\det( P^{ -1}){\det(\adj(\bar{A}))}\succmod\det(A)^{ n-1},\\
\det(A\adj(A))&=&\det(\det(P)\bar{A} P P^{ -1}\adj(\bar{A}))\\
&=&\det(A)^n{\det(\II_{\bar{A}})}\succmod\det(A)^{ n},\\
\det(A^\nabla)&=&\det(A)^{-n}\det(\adj(A))\succmod\det(A)^{ -1}\enspace.
\qed
\end{eqnarray*}
\renewcommand{\qedsymbol}{}
\end{proof}

Recalling Proposition~\ref{qasinv} and Theorem~\ref{detAB}, the matrices~
$\II_A$ and ${A}^\nabla$ in~\eqref{cor341} are
definite over~$\supertropical$, as proved in~\cite[Remark~2.18]{PI&CP}, and
the~$\succmod$ relations in~\eqref{cor341} and~\eqref{cor342} become equalities over~$\supertropical$, as proved 
in~\cite[Theorem 4.9]{STMA}.
Note that Point~\eqref{cor342} of Proposition~\ref{defbar}
is an  equality when $\TT=\Mod$ and in particular when $\TT=\rmax$
and $\det(A)\in \TT^*$. This implies that
$|\det(A^\nabla)|=|\det(A)|^{-1}$ holds in $\TT=\supertropical$ 
for any matrix $A$, either singular or not, such that $\det(A)\neq \zero$.

We provide one more identity concerning the quasi-inverse, to be used in Section~\ref{epcm}.
\begin{pro}\label{nabcom}
If~$A$ is nonsingular, then
$A^{\wedge k}(A^\nabla)^{\wedge k}\succeq^\circ \mathcal{I}.$\end{pro}
(Not to be confused with~$A^{\wedge k}(A^{\wedge k})^\nabla$ and~$(AA^\nabla)^{\wedge k}$ 
which are obviously~$\succeq^\circ\mathcal{I}$.)
\begin{proof}Let~$\bar{A}P$ be a right normalization of~$A$. 
By Corollary~\ref{eqadj}, we have 
$$A^{\wedge k}(A^\nabla)^{\wedge k}=\bar{A}^{\wedge k}P^{\wedge k}(P^{-1})^{\wedge k}
(\bar{A}^\nabla)^{\wedge k}=\bar{A}^{\wedge k}P^{\wedge k}(P^{\wedge k})^{-1}
(\bar{A}^\nabla)^{\wedge k}=\bar{A}^{\wedge k}(\bar{A}^\nabla)^{\wedge k},$$ 
then it is sufficient to 
prove the identity for~$A$ definite.
Using Theorem~\ref{QDF} and the fact that $D^{\wedge k}$ is a diagonal matrix
equal to its inverse,
we get
\begin{eqnarray}
\nonumber \big(A^{\wedge k}(A^\nabla)^{\wedge k}\big)_{I,J}
&=& \bigoplus_{L\in\mathcal{P}_k([n])} A^{\wedge k}_{I,L}(A^\nabla)^{\wedge k}_{L,J}\\
\nonumber &=& \bigoplus_{L\in\mathcal{P}_k([n])} A^{\wedge k}_{I,L}
D^{\wedge k}_{L,L} (DA^\nabla D)^{\wedge k}_{L,J}D^{\wedge k}_{J,J} \\
&\succeq^\circ&
 \bigoplus_{L\in\mathcal{P}_k([n])} A^{\wedge k}_{I,L}
D^{\wedge k}_{L,L} A^{\wedge n-k}_{J^c,L^c} D^{\wedge k}_{J,J} \enspace . \label{knabla0}
\end{eqnarray}
The right hand side of~\eqref{knabla0} is the sum of all
expressions of the form
\begin{equation}\label{knabla}\sign(\sigma)\bigodot_{i\in I}a_{i,\sigma(i)}(\ominus\unit)^{\sum_{i\in L}i+\sum_{j\in J}j}
\sign(\tau)\bigodot_{j\in J^c}a_{j,\tau(j)} \enspace ,\end{equation} 
where $\sigma\in\allperm_{I,L}$ , $\tau\in\allperm_{J^c,L^c}$.

If~$I\cap J^c=\emptyset$, then~$I= J$ and one can extend the bijection~$\sigma$ into a 
permutation~$\rho\in\allperm_{[n]}$ such that~$\rho|_{I}=\sigma$ and~$\rho|_{J^c}=\tau$. 
Using Proposition~\ref{resign}, we obtain that~\eqref{knabla} is equal to~$\sign(\rho)
\bigodot_{i\in[n]}a_{i,\rho(i)}$,  so the right hand side of~\eqref{knabla0}
is equal to~$\det(A)=\unit$. 

If~$I\cap J^c\ne\emptyset$, denote
 $I=\{i_1<\dots< i_{k}\}$ and $J=\{j_1<\dots <j_{k}\}$
and define $\delta\in \allperm_{J, I}$ by $\delta(j_t)=i_t$,
and let $\pi=\sigma\circ\delta\in\allperm_{J, L}$.
Then, $\sign(\delta)=1$ and, by Property~\ref{signs},
$\sign(\pi)=\sign(\sigma)$.
Let $\rho\in\allperm_{[n]}$ be such that~$\rho|_{J}=\pi$ and~$\rho|_{J^c}=\tau$. 
From Proposition~\ref{resign} again, we have that
\begin{equation}\label{knabla1}
 \sign(\rho)=
\sign(\sigma)\sign(\tau)(\ominus\unit)^{\sum_{j\in J}j+\sum_{i\in L}i}\enspace.\end{equation} 
Let $m\in I\cap J^c$.  Define $L'=\{\tau(m)\}\cup L\setminus\{\sigma(m)\}\in\mathcal{P}_k([n])$ 
and~$\sigma'\in\allperm_{I,L'},\tau'\in\allperm_{J^c,L'^c}$ by
$\sigma'|_{\{m\}^c}=\sigma|_{\{m\}^c},\ 
\sigma'(m)=\tau(m),\ \tau'|_{\{m\}^c}=\tau|_{\{m\}^c},\ \tau'(m)=\sigma(m).$
Next, consider $\rho'\in\allperm_{[n]}$ such that~$\rho'|_{J}=\sigma'\circ\delta$ and~$\rho'|_{J^c}=\tau'$. 
Since $\rho'=\big(\sigma(m)\ \tau(m)\big)\circ \rho$, we have that 
$$\ominus\sign(\rho)=\sign(\rho')=\sign(\sigma')\sign(\tau')(\ominus\unit)^{\sum_{j\in J}j+\sum_{i\in L'}i}.$$
Moreover, 
$$\bigodot_{i\in I}a_{i,\sigma(i)}\bigodot_{j\in J^c}a_{j,\tau(j)} =
\bigodot_{i\in I}a_{i,\sigma'(i)}\bigodot_{j\in J^c}a_{j,\tau'(j)}
\enspace .$$
 Therefore, 
\eqref{knabla} reappears in the right hand side of~\eqref{knabla0}, with~$\sigma,\tau$ and~$L$ replaced by~$\sigma',\tau'$ and~$L'$  respectively,
and with an opposite sign.
\end{proof}

\subsection{Powers of matrices}\label{eppm}
Let~$M$ be the weight matrix of a weighted directed graph~$G$. Considering a permutation of~${[n]}$ 
in~$G$, we analyze the corresponding permutation in the graph having~$M^{ m}$ as its weight matrix.  As one 
can see in Theorem~\ref{DF}, the elementary paths in graphs having powers of matrices as  a weight matrix, 
satisfy rather unusual characteristics. In this section we provide two  additional properties. The first is an analogue 
to a classical property, and the second  holds over~$\supertropical$, but neither classically,  over~$\smax$ or 
over~$\mathbb{R}_{\max}$, as shown in the following example. Nevertheless,  the second property leads to 
an analogue of a classical result, stated in Corollary~\ref{CPNS}.

\begin{exa}
 $$\text{If }\ A=\left(\begin{array}{cc}a&b\\c&d\end{array}\right),\ \text{ then }\ A^2=
\left(\begin{array}{cc}a^2\oplus bc&b(a\oplus d)\\c(a\oplus d)&d^2\oplus bc\end{array}\right),$$ and we get
\begin{equation}\label{exapowers}\tr\big(A^2\big)=\begin{rcases}\begin{cases}a^2\oplus d^2\oplus bc\oplus bc\ 
=\ a^2\oplus d^2&\text{over a field of characteristic }2\\a^2\oplus d^2\oplus bc\ \ \succeq\  a^2\oplus d^2&
\text{over }\smax\text{ or }\mathbb{R}_{\max}\\a^2\oplus d^2\oplus bc^\nu\ \models\ a^2
\oplus d^2&\text{over }\supertropical\end{cases}\end{rcases}=\tr(A)^2.\end{equation}

\end{exa}

 The first identity is the elementary consequence~$\tr(A^p)=\tr(A)^p$ of  Frobenius  property over a field 
of characteristic~$p$.
 In the following  theorem and corollary, we extend an identity proved in~\cite[Theorem~3.6]{PCP}, and provide 
the supertropical property motivated by the third identity in~\eqref{exapowers}.

\begin{thm}\label{PEP} Let $A\in\TT^{n\times n}$,        $m\in \mathbb{N}$ and $k\in\{0,\ldots, n\}$.
\begin{enumerate}
\item\label{thm1} We have $(A^{ m})^{\wedge k}\succeq^\circ \big(A^{\wedge k}\big)^{ m}.$

\item\label{thm2} If $\TT=\Mod$ or $\TT=\supertropical$, 
we have $\tr\big((A^{\wedge k})^{ m}\big)\succeq^\circ \big(\tr(A^{\wedge k})\big)^{ m}.$
\end{enumerate}\end{thm}

\begin{proof}
\eqref{thm1} is obtained by induction on $m$, using Proposition~\ref{CBF} applied with $B=A^{m-1}$ and using the compatibility of $\succeq^\circ$ with the laws of $\TT$.

For~\eqref{thm2}, we show   more generally that~$\tr(B^m)\succeq^\circ (\tr(B))^m$ for any $B\in \TT^{n\times n}$ when $\TT=\Mod$ or $\TT=\supertropical$. We have
\begin{equation}\label{powcomp}\tr(B^m)=\bigoplus_{i\in[n]}(B^m)_{i,i}=\bigoplus_{i\in[n]}
\bigoplus_{\substack{t_{i,l}\in[n]\\\ell\in[m-1]}}B_{i,t_{i,1}}B_{t_{i,1},t_{i,2}}\cdots B_{t_{i,m-1},i},\end{equation}
and using Theorem~\ref{FP}, we have
\begin{equation}\label{compow}(\tr(B))^m=\bigoplus_{i\in[n]}B_{i,i}^m.\end{equation}
The summing terms in the right hand side of~\eqref{compow} are also terms in the right hand side of~\eqref{powcomp} with $t_{i,\ell}=i,\ \forall \ell\in[m-1],\ \forall i\in[n]$.
For every other term of~\eqref{powcomp}, there exists~$\ell$ s.t.~$t_{i,\ell}\ne i,$ and therefore this term reappears as a term of $(B^m)_{t_{i,\ell},t_{i,\ell}}$:
$$B_{i,t_{i,1}}\cdots B_{t_{i,\ell-1},t_{i,\ell}}\ B_{t_{i,\ell},t_{i,\ell+1}} \cdots B_{t_{i,m-1},i}=B_{t_{i,\ell},t_{i,\ell+1}} 
\cdots B_{t_{i,m-1},i}\ B_{i,t_{i,1}}\cdots B_{t_{i,\ell-1},t_{i,\ell}}\enspace.$$
Since $\ominus \unit=\unit$ in $\TT$, we deduce $\tr(B^m)\succeq^\circ (\tr(B))^m$.
\end{proof}

Using the compatibility of $\succeq^\circ$ with the laws of $\TT$, we deduce the following result.
\begin{cor}\label{compowtr} Let $\TT=\Mod$ or $\TT=\supertropical$, 
$A\in\TT^{n\times n}$, $m\in \mathbb{N}$ and $k\in\{0,\ldots, n\}$. We have:
$$\tr\big((A^{ m})^{\wedge k}\big)\succeq^\circ \big(\tr(A^{\wedge k})\big)^{ m}
\enspace .$$
\end{cor}
Note that the previous inequality can be written as:
\begin{align*}
\tr\big((A^{ m})^{\wedge k}\big)\models \big(\tr(A^{\wedge k})\big)^{ m}
&\quad\text{over}\; \supertropical\enspace ,\\
\tr\big((A^{ m})^{\wedge k}\big)\succeq \big(\tr(A^{\wedge k})\big)^{ m}
&\quad\text{over} \;\Mod \enspace .
\end{align*}

In Section~\ref{tcp} we shall apply Corollary~\ref{compowtr} to the  characteristic polynomials of~$A$ and 
its powers, and provide over $\supertropical$ an analogue to the property: 
if~$\lambda$ is an eigenvalue of~$A$, then~$\lambda^m$ is an eigenvalue of~$A^m$, which holds over rings.


\subsection{Conjugate matrices}\label{epcm}

\begin{df} We say that a matrix~$A$ is tropically conjugate to~$A'$ if there exists a nonsingular matrix~$E$ 
such that $A'=E^\nabla A E.$\end{df}
Tropical conjugation is not a symmetric relation, as shown in the following example.

\begin{exa} We take
$A=\mathcal{I},\ E=\left(
\begin{array}{cc}
\unit & \unit\\
\zero & \unit
\end{array}
\right)\ \text{ and }\ A'=\mathcal{I}'_E\ \text{ in }\ \supertropical^{^{2\times 2}}.$
 
Obviously~$E^\nabla A E=A'$. Exchanging the roles of $A$ and $A'$, we look for a nonsingular matrix~$F=\left(
\begin{array}{ccc}
a & b\\
c & d
\end{array}
\right)$~s.t. $F^\nabla A' F= A$, which means that
$$\adj(F) A'F=\left(
\begin{array}{cc}
d & b\\
c & a
\end{array}
\right) \mathcal{I}'_E
\left(
\begin{array}{ccc}
a & b\\
c & d
\end{array}
\right)=\left(
\begin{array}{ccc}
a d\oplus b c\oplus c^\nu d & (b d\oplus d^{ 2})^\nu\\&\\
(c a\oplus c^{ 2})^\nu & a d\oplus b c\oplus d^\nu c
\end{array}
\right)=\det(F)\mathcal{I}.$$ In order for the~$(1,2)$ and~$(2,1)$ positions to be~$\zero$, we must require~$c=d=\zero$, 
which means~$\det(F)=\zero$, and therefore such a nonsingular matrix does not exist.
\end{exa}

Conjugate matrices have algebraic value due to the various properties being preserved. In particular, it is a 
key tool in representation theory for describing equivalent representations (see~\cite{RT}), and in linear algebra 
for diagonalizing, triangularizing or Jordanizing matrices (see~\cite{ILA}). 
In the present section we are interested in a well known identity on the  compound matrix of a  conjugation.  
Whereas over fields  the proof is straightforward due to the multiplicativity of the compound operation, in the tropical case  
this property becomes an inequality, unless the conjugating matrix is invertible. Nevertheless,   due to the invertible case, 
we can reduce the desired property to definite matrices.

\begin{lem}\label{CDF} Let~$E,A\in \TT^{n\times n}$. If $E$ is nonsingular, with right definite form~$\bar{E}$, then
\[\tr\big((E^\nabla A E)^{\wedge k}\big)=\tr\big((\bar{E}^\nabla A\bar{E})^{\wedge k}
\big).\]
\end{lem}

\begin{proof}  We recall that the trace function satisfies $\tr(MN)=\tr(NM)$ for any two square matrices $M,N$.
 Let $P$ be the right normalizer of $E$, that is $E=\bar{E} P$.
Using Corollary~\ref{eqadj}, we get
$$ \tr\big((E^\nabla A E)^{\wedge k}\big)=\tr\big((P^{ -1} \bar{E}^\nabla A\bar{E} P)^{\wedge k}\big)=
\tr\big((P^{ -1})^{\wedge k}(\bar{E}^\nabla A\bar{E})^{\wedge k}(P)^{\wedge k}\big)=$$
$$\tr\big((P)^{\wedge k}(P^{ -1})^{\wedge k} (\bar{E}^\nabla A\bar{E})^{\wedge k}\big)=
\tr\big((\bar{E}^\nabla A\bar{E})^{\wedge k}\big). \qed$$
\renewcommand{\qedsymbol}{}
\end{proof}

\begin{thm}\label{CEP}
Let $E,A\in \TT^{n\times n}$. If $E$ is nonsingular, then $$\tr\big((E^\nabla A E)^{\wedge k}\big)
\succeq^\circ \tr\big(A^{\wedge k}\big),\ \forall k\in[n].$$
\end{thm}
\begin{proof}
Using Proposition~\ref{CBF}, 
the compatibility of $\succeq^\circ$ with the laws of $\TT$,
the commutativity of the trace, and Proposition~\ref{nabcom}, we obtain
$$\tr\big((E^\nabla A E)^{\wedge k}\big)\succeq^\circ \tr\big((E^\nabla)^{\wedge k} A^{\wedge k} 
E^{\wedge k}\big)=\tr\big( \underbrace{E^{\wedge k}(E^\nabla)^{\wedge k}}_{\succeq^\circ \mathcal{I}} 
A^{\wedge k}\big)\succeq^\circ\tr(A^{\wedge k}).\qed $$
\renewcommand{\qedsymbol}{}\end{proof}

In Section~\ref{tcp}  we apply Theorem~\ref{CEP} to the characteristic polynomials of $A$ and 
its conjugates, which provide the desired connection between their eigenvalues.


\section{The Sylvester--Franke identity}\label{sf}
The Sylvester--Franke identity has received a few
 proofs over the years, including diagonalization (see~\cite{UI}) and factorization to elementary matrices  
(see~\cite{SFT}). The tropical version for this identity is provided in this section combinatorially, proving equality over~$\mathbb{R}_{\max}$
and $\supertropical$ in particular.

\begin{thm}[Tropical Sylvester--Franke theorem] \label{TSF} 
 Let $A\in \TT^{n\times n}$,
and $k\in\{0,\ldots, n\}$. The identity 
\begin{equation}\label{eqSF}
\det(A^{\wedge k})=\det(A)^{\left(\substack{n-1\\k-1}\right)}\enspace , 
\end{equation}
holds under one of the following conditions:
\begin{enumerate}
\item \label{TSF1} $\TT=\Mod$, 
in which case $\det$ coincides with $\per$;
\item\label{TSF1bis} $A$ is invertible in $\TT$;
\item\label{TSF2} $A=I\ominus B$ is definite with $B_{ii}=\zero$ for all
$i\in [n]$ and the weight of every cycle in $B$ is  $\curlyeqprec\unit$
(which holds in particular
 when the modulus of the weight of every cycle in $B$ is  
strictly dominated by $\unit$).
\end{enumerate}
\end{thm}
\begin{proof}
 Let $A=(a_{i,j})\in \TT^{n\times n}$,
and let $A^{\wedge k}=(A^{\wedge k}_{I,J})$ be its~$k$th compound matrix. 

\eqref{TSF1} Assume $\TT=\Mod$.
Then, $\det$ coincides with $\per$. We have 
\begin{equation}\label{detcomp}\per(A^{\wedge k})=\bigoplus_{\pi\in\allperm_{\mathcal{P}_k([n])
}}\bigodot_{I\in \mathcal{P}_k([n])}A^{\wedge k}_{I,\pi(I)}=
\bigoplus_{\pi\in\allperm_{\mathcal{P}_k([n])
}}\bigodot_{I\in \mathcal{P}_k([n])}
\left(\bigoplus_{\sigma\in\allperm_{I,\pi(I)}}\bigodot_{i\in I}a_{i,\sigma(i)}\right)\enspace ,\end{equation} and
\begin{equation}\label{powdet1}\per(A)^{\left(\substack{n-1\\k-1}\right)}=
\left(\bigoplus_{\rho\in\allperm_{[n]}} \bigodot_{i\in[n]}a_{i,\rho(i)}\right)^{\left(\substack{n-1\\k-1}\right)} 
\enspace .\end{equation}
Moreover, using Theorem~\ref{FP}, we also have
\begin{equation}\label{powdet}\per(A)^{\left(\substack{n-1\\k-1}\right)}=
\bigoplus_{\rho\in\allperm_{[n]}}
\big(\bigodot_{i\in[n]}a_{i,\rho(i)}\big)^{\left(\substack{n-1\\k-1}\right)}.\end{equation}
Developing~\eqref{detcomp},~\eqref{powdet1} and~\eqref{powdet}
using the distributivity of the multiplication with respect to addition,
we arrive at a sum such that
 each summand is a product of $\underbrace{\left(\substack{n\\k}\right)\cdot k}_{\text{in~\eqref{detcomp}}}=
\underbrace{\left(\substack{n-1\\k-1}\right)\cdot n}_{\text{in~\eqref{powdet1} and~\eqref{powdet}}}$  entries of $A$.
We shall show that each summand in~\eqref{detcomp} is a summand 
in~\eqref{powdet1}, and that each summand in~\eqref{powdet} is a summand 
in~\eqref{detcomp}.

For each summand in~\eqref{detcomp},~\eqref{powdet1} or~\eqref{powdet},
we shall consider the $n\times n$ integer matrix 
$B=(b_{i,j})$ such that $b_{i,j}$ is the number of occurrences of the factor
$a_{i,j}$ in the summand, which means that the summand is equal to 
$\bigodot_{i,j\in[n]}a_{i,j}^{b_{i,j}}$.
Equivalently, one may consider the multigraph with set of nodes $[n]$,
and $b_{i,j}$ arcs between $i$ and $j$, for each $i,j \in [n]$.
From the above remark, $B$ satisfies necessarily 
$\sum_{i,j\in [n]} b_{i,j}= \left(\substack{n\\k}\right)\cdot k$.

\underline{Summands in~\eqref{detcomp} are summands in~\eqref{powdet1}:} 
Let $B$ be an integer matrix associated to a summand in~\eqref{detcomp}.
For every~$m\in[n]$, the number of sets~$I\in \mathcal{P}_k([n])$ such that~$m\in I$, 
is~$\left(\substack{n-1\\k-1}\right)$.
Thus, $\sum_{j\in [n]} b_{m,j}=\left(\substack{n-1\\k-1}\right)$.
By symmetry, that is indexing the product in~\eqref{detcomp} by the
image $J$ of $\pi$, we get that
$\sum_{i\in [n]} b_{i,m}=\left(\substack{n-1\\k-1}\right)$ for all $m\in [n]$.
Using the Birkhoff-von Neumann theorem  (or Hall's Theorem,
see e.g.\ Hall, {\cite[Theorem 5.1.9]{hall}}), 
we get that $B$ can be written as the sum of 
$\left(\substack{n-1\\k-1}\right)$ permutation matrices
(in the usual sense).
Since the matrix associated to the permutation  $\rho\in\allperm_{[n]}$
corresponds to $\bigodot_{i\in[n]}a_{i,\rho(i)}$, we obtain
that $B$ is the integer matrix associated to the product of
$\left(\substack{n-1\\k-1}\right)$ such products.
This means that the summand of~\eqref{detcomp} considered initially
is also a summand in~\eqref{powdet1}.

\underline{Summands in~\eqref{powdet} are summands in~\eqref{detcomp}:} 
Let $B$ be the integer matrix associated to a summand in~\eqref{powdet}.
We have $B=\left(\substack{n-1\\k-1}\right) P$, where $P$ is the
matrix  of some permutation $\rho\in\allperm_{[n]}$.
Consider $\pi=\rho^{(k)}\in\allperm_{\left[\mathcal{P}_k([n])\right]}$, and,
for all $I\in \mathcal{P}_k([n])$, take $\sigma=\rho|_I$ in~\eqref{detcomp}.
We get that $\sigma\in\allperm_{I,\pi(I)}$ and
$\sigma(i)=\rho(i)$ for all $i\in I$.
Since the number of sets $I\in \mathcal{P}_k([n])$ such that~$i\in I$
is equal to $\left(\substack{n-1\\k-1}\right)$, we obtain that 
$\bigodot_{i,j\in[n]}a_{i,j}^{b_{i,j}}= \left(\bigodot_{i\in[n]}a_{i,\rho(i)}\right)^{_{\left(\substack{n-1\\k-1}\right)}}=\bigodot_{I\in \mathcal{P}_k([n])}
\left(\bigodot_{i\in I} a_{i,\rho|_I(i)}\right)$, so that
the summand of~\eqref{powdet} considered initially
is also a summand in~\eqref{detcomp}.

Now since $\TT=\Mod$, $\TT$ is totally ordered, thus idempotent.
Hence,  $\per(A^{\wedge k})\preceq \per(A)^{\left(\substack{n-1\\k-1}\right)}$
because summands in~\eqref{detcomp} are summands in~\eqref{powdet1},
and $\per(A)^{\left(\substack{n-1\\k-1}\right)} \preceq \per(A^{\wedge k})$
because summands in~\eqref{powdet} are summands in~\eqref{detcomp}.
This shows~\eqref{eqSF}.

\eqref{TSF1bis} 
Assume that $A$ is invertible in $\TT$. Then, $A$ is a monomial matrix.
From Remark~\ref{compinv}, for a monomial matrix $A$ on any semiring $\SS$,
$A^k$ is a monomial matrix. This implies that in that case,
both sides of~\eqref{eqSF} are monomials in the non-zero entries of $A$
with a coefficient equal to $\unit$ or $\ominus\unit$,
depending on the sign of the permutation associated to $A$, on $k$ and $n$.
Since the equality in~\eqref{eqSF} holds on any commutative ring,
it also holds on any semiring $\TT$ for a monomial matrix
(see for instance~\cite{reutstraub} or~\cite{LDTS}).

\eqref{TSF2}  
Assume that $A=I\ominus B$ is definite with $B_{ii}=\zero$ for all
$i\in [n]$ and that the weight of every cycle in $B$ is  $\curlyeqprec\unit$.
Recall that this condition holds in particular
 when the modulus of the weight of every cycle in $B$ is  
strictly dominated by $\unit$, see Theorem~\ref{DF}.
Also by Theorem~\ref{DF}, under the above conditions, we also have that
the weight of every cycle in $B$ is $\curlyeqprec\ominus\unit$.
Since $\det(A)=\unit$, we only need to show that $\det(A^k)=\unit$.
We have
\begin{eqnarray}
\det(A^{\wedge k})&=&\bigoplus_{\pi\in\allperm_{\mathcal{P}_k([n])}}\sign(\pi)\bigodot_{I\in \mathcal{P}_k([n])}A^{\wedge k}_{I,\pi(I)}\nonumber\\
&=&
\bigoplus_{\pi\in\allperm_{\mathcal{P}_k([n])}}\sign(\pi)\bigodot_{I\in \mathcal{P}_k([n])}
\left(\bigoplus_{\sigma\in\allperm_{I,\pi(I)}}\sign(\sigma)\bigodot_{i\in I}a_{i,\sigma(i)}
\right)\enspace .
\label{detcomp2}\end{eqnarray}
By the arguments above, each summand in~\eqref{detcomp2}
is equal to a summand in~\eqref{powdet1} times $\unit$ or $\ominus\unit$,
thus it is equal to the product of $\left(\substack{n-1\\k-1}\right)$
signed permutations of $A$ times $\unit$ or $\ominus\unit$.
decomposing permutations into cycles, and using that $A=I\ominus B$
and $B_{i,i}=\zero$, we get that if one of these permutations is not 
equal to the identity, the summand is equal to 
a (nonempty) product of  weights of nontrivial cycles in $B$ times 
$\unit$ or $\ominus\unit$.
In that case, using that the weight of every cycle in $B$ is 
 $\curlyeqprec\unit$, and also  $\curlyeqprec\ominus\unit$,
we deduce that the summand is $\curlyeqprec\unit$.
Otherwise, if all the permutations are equal to the identity, 
the summand corresponds to the permutation
$\pi$ equal to identity and all bijections $\sigma$ equal to the identity, 
in which case the summand is equal to $\unit$,  
since all diagonal entries of $A$ are equal to $\unit$.
In all, this imply that~\eqref{detcomp2} is equal to $\unit$ and so
\eqref{eqSF} holds.
\end{proof}

\begin{cor}
\label{TSF-cor} Let $A\in \TT^{n\times n}$,
and $k\in\{0,\ldots, n\}$. We have
\begin{equation}\label{eqSFmod}
|\det(A^{\wedge k})|=|\det(A)|^{\left(\substack{n-1\\k-1}\right)}\enspace . 
\end{equation}
Moreover,~\eqref{eqSF} holds under one of the following conditions:
\begin{enumerate}
\item \label{TSF-cor1} $\ominus \unit=\unit$ and $A$ is nonsingular;
\item \label{TSF-cor2} $\TT=\supertropical$, $A$ being not necessarily
nonsingular.
\end{enumerate}
\end{cor}
\begin{proof}
The first assertion follows from Point~\eqref{TSF1} of Theorem~\ref{TSF} 
applied to $|A|$, and the property that the modulus is a morphism.

\eqref{TSF-cor1} Let $A$ be nonsingular, and let $A=P\bar{A}$
be any definite form. Using Corollary~\ref{eqadj}, and Point~\eqref{TSF1bis}
of Theorem~\ref{TSF}, we obtain that~\eqref{eqSF} holds for $A$ as soon as
it holds for the definite matrix $\bar{A}$.
When $A$ is definite, writing it as $A=I\ominus B$ as in 
Theorem~\ref{DF}, and using
Theorem~\ref{DF} and $\unit=\ominus \unit$, we get that
the weight of every cycle in $B$ is  $\curlyeqprec\unit$.
Then, by Point~\eqref{TSF2} of Theorem~\ref{TSF},
we deduce that~\eqref{eqSF} holds for $A$.

\eqref{TSF-cor2}  Assume $\TT=\supertropical$.
If $A$ is nonsingular, then~\eqref{eqSF} holds by the previous point 
since $\ominus \unit=\unit$.
Assume now that $A$ is singular. 
Let $a=\det(A^k)$ and 
$b=\det(A)^{\left(\substack{n-1\\k-1}\right)}$ be the left and right hand sides
of~\eqref{eqSF}, respectively.
From the first assertion of the present corollary, we have $|a|=|b|$.
Since  $A$ is singular, we get that $b\in\TT^\circ$ (that is $b\in 
\R^\nu\cup\{-\infty\}$), thus $b$ is the maximal element of $\TT$ 
with modulus equal to $|b|$, which implies that $a\preceq b$.

Now, since $\det$ coincides with $\per$, and
Theorem~\ref{FP} holds for $\TT=\supertropical$, 
Equalities~\eqref{detcomp},~\eqref{powdet1} and~\eqref{powdet} hold true.
Also, the implications between summands shown 
in the proof of  Point~\eqref{TSF1} of Theorem~\ref{TSF} are 
also valid.
In addition, the correspondence between summands in~\eqref{powdet} 
and summands in~\eqref{detcomp} constructed in that proof is one to one.
Indeed, if $\rho\neq \rho'$, then $\rho^{(k)}\neq (\rho')^{(k)}$.
This implies that the sum in~\eqref{powdet} is $\preceq$ to the sum 
in~\eqref{detcomp}, that is $b\preceq a$, so $b=a$,
which finishes the proof of~\eqref{eqSF}.
\end{proof}

\begin{exa}
Note that the condition $\ominus\unit=\unit$ in Point~\eqref{TSF-cor1}
of Corollary~\ref{TSF-cor} is necessary.
Indeed, let $\TT=\smax$  and
\[ A=\left(\begin{array}{ccc}\unit&\ominus \unit & \zero\\
\unit&\unit&\ominus \unit\\
\unit&\unit&\unit\end{array}\right)\enspace .\]
Then $\det(A)=\unit$ so that $A$ is nonsingular. However,
\[ A^{\wedge 2}=\left(\begin{array}{ccc}\unit&\ominus \unit & \unit\\
\unit&\unit&\ominus \unit\\
\unit^\circ&\unit&\unit\end{array}\right)\enspace ,\]
and $\det( A^{\wedge 2})=\unit^\circ\neq \det(A)^2 $.
\end{exa}


\section{The tropical characteristic polynomial}\label{tcp}
The  characteristic polynomial over 
the symmetrized and supertropical semirings are studied in~\cite{LDTS} 
and~\cite{STLA}, respectively, whose terminology we follow. 
In this section, we investigate tropical characteristic polynomials of a matrix, its powers, its quasi-inverse and its
 conjugates, using the results of Sections~\ref{jacobi} and~\ref{oid}. We obtain an analogue to  properties  connecting   the eigenvalues of these matrices, which are wider in the extended tropical semiring,
and we address  the special case where the coefficients of the characteristic polynomial are invertible.

Let us first give some definitions.
From Properties~\ref{prop1} and~\ref{prp1}, 
a polynomial over~$\TT$ takes generally the value of monomials of highest absolute value. Yet some 
monomials do not dominate for any~$x\in\TT$.

\begin{df}
Let $f=\oplus_{k=0}^n a_k X^k \in \TT[X]$  be a formal polynomial over $\TT$ and let $f(x)$ be its 
evaluation in $x\in\TT$.
 We call monomials $a_k x^k$ that dominate~$f(x)$ at some $x\in \TT$
(that is such that $|a_k x^k|=|f(x)|$) \myemph{essential} at $x$, and 
monomials that do not dominate~$f(x)$  for any~$x\in \TT$ \myemph{inessential}. 

 We call an element~$r \in \TT^\vee$ a \myemph{root} of~$f$, if~$f(r)\in \TT^\circ$. 
If $r$ is a root such that either $r\in\TT^*$ and
$f(r)$ is the sum of at least two essential monomials at $r$ that have nonsingular coefficients, that is, there exists
a set $S\subset\{0\leq k\leq n\mid |a_k r^k|=|f(r)|\; \text{and}\; a_k\in \TT^*\}$
with at least two elements,
such that $f(r)=\oplus_{k \in S} a_k r^k$,
or $r=\zero$ and $f(r)=\zero$, then $r$ 
will be called a \myemph{corner root}.
Roots that are not corner roots will be called \myemph{non-corner roots}.
 \end{df}
On can also give the following different definitions of a corner root.
\begin{lem}\label{lem-corn}
Let $r\in\TT^*$ be a root of the polynomial
$f=\oplus_{k=0}^n a_k X^k \in \TT[X]$. Then, the following are equivalent:
\begin{enumerate}
\item\label{corn1}   $r$ is a corner root  of $f$;
\item\label{corn2}   $S=\{0\leq k\leq n \mid |a_k r^k|=|f(r)|\; \text{and}\; a_k\in \TT^*\}$
has at least two elements and $f(r)=\oplus_{k \in S} a_k r^k$.
\end{enumerate}
If $\TT\neq \TT^\circ$, they are also equivalent to:
\begin{enumerate}
\addtocounter{enumi}{2}
\item\label{corn3}  $f(r)=\oplus_{k \in S} a_k r^k$ with $S=\{0\leq k\leq n \mid |a_k r^k|=|f(r)|\; \text{and}\; a_k\in \TT^*\}$;
\item\label{corn4}  $f(r)=\oplus_{k \in S} a_k r^k$ with $S=\{0\leq k\leq n \mid a_k\in \TT^*\}$;
\item\label{corn5}  $f(r)=\oplus_{k \in S} a_k r^k$ for some set 
$S\subset\{0\leq k\leq n \mid a_k\in \TT^*\}$.
\end{enumerate}
Moreover, an invertible root such that all essential monomials
have nonsingular coefficients is necessarily a corner root.
\end{lem}
\begin{proof}
The implications \eqref{corn2}$\Rightarrow$\eqref{corn1} and 
\eqref{corn1}$\Rightarrow$\eqref{corn5} are trivial.

\eqref{corn1}$\Rightarrow$\eqref{corn2}: let
$S\subset S':=\{0\leq k\leq n \mid |a_k r^k|=|f(r)|\; \text{and}\; a_k\in \TT^*\}$ such that $f(r)=\oplus_{k \in S} a_k r^k$ and $S$ 
has at least two elements. Then, $S'$ has at least two elements and
since 
 $f(r)=\oplus_{k \in S} a_k r^k\preceq \oplus_{k \in S'} a_k r^k\preceq f(r)$,
we deduce that $f(r)=\oplus_{k \in S'} a_k r^k$, which implies~\eqref{corn2}.

\eqref{corn5}$\Rightarrow$\eqref{corn4} follows from the same arguments as 
for the previous implication.

\eqref{corn4}$\Rightarrow$\eqref{corn3} follows from Property~\ref{prop1}.

Assume now that $\TT\neq \TT^\circ$.

\eqref{corn3}$\Rightarrow$\eqref{corn1} 
Assume that $r\in\TT^*$ is a root of $f$, and 
that $f(r)=\oplus_{k \in S} a_k r^k$ with $S=\{0\leq k\leq n \mid |a_k r^k|=|f(r)|\; \text{and}\; a_k\in \TT^*\}$.
Since $r\in\TT^*$, we have $a_k r^k\in \TT^*$ for all $k\in S$, 
and since $r$ is a root of $f$, so $f(r)\in\TT^\circ$, and
$\TT\neq \TT^\circ$, so $\TT^*$ and $\TT^\circ$ are disjoint,
we get that $f(r)\neq a_k r^k$ for all $k\in S$. Therefore, $S$ 
has at least two elements, and $r$ is a corner root.

By Property~\ref{prop1}, $f(r)$ is the sum of all essential monomials
of $f$ at $r$, so if all essential monomials
have nonsingular coefficients, then $r$ satisfies~\eqref{corn3},
which implies the last assertion of the lemma.
\end{proof}

When $\TT=\TT^\circ$, a root is simply any element of $\TT^\vee$.
However, for $\TT=\rmax$ for instance, $r\neq \zero$ is a corner root
if and only if the maximum in $f(r)=\oplus_{k=0}^n a_k r^k$ is attained at 
least twice. Then, corner roots over $\rmax$ coincide with roots
of tropical polynomials in the sense of tropical geometry.
If now $\Mod=\rmax$,
in particular for $\TT=\smax$ or $\TT=\supertropical$,
a corner root $r$ of $f$ is such that $|r|$ is a corner root 
(or a root in the sense of tropical geometry) of $|f|$,
 that is either $|r|=\zero$ and $|a_0|=\zero$,
or $|r|\neq \zero$ and the maximum in the expression 
$\max(|a_k| |r|^k, \; k=0,\ldots, n)$ is attained at least twice.
The converse is not true in general.
For instance over $\TT=\smax$, the polynomial $f=\unit\oplus X^2$ is such that
 $|f|=\unit\oplus X^2$ over $\rmax$, so $\unit$ is a corner root of $|f|$ (with
multiplicity 2), but the only elements $r\in\smax^\vee$ such that 
$|r|=\unit$ are $\oplus \unit$ and $\ominus \unit$ and both 
satisfy $f(r)=\unit\not\in\TT^\circ$, so there exist no (corner) roots of $f$
such that $|r|$ is a corner root of $|f|$.
Also on $\TT=\supertropical$, $f=\unit^\nu\oplus X$ is such that
$|f|=\unit\oplus X$ over $\rmax$, so $\unit$ is a corner root of $|f|$,
but the only element $r\in\supertropical^\vee$ such that 
$|r|=\unit$ is $\unit$, which is a non-corner root of $f$.

For the next definition,  we follow the terminology  in~\cite{TBE,MPA16} 
and~\cite{STLA}.

\begin{df}\label{cp} The \myemph{(formal) characteristic polynomial} of~$A\in\TT^{n\times n}$  is defined to be $$f_A=
\det(X \mathcal{I}\ominus A)\in \TT[X]
\enspace,$$
and its \myemph{characteristic polynomial function} 
is $f_A(x) = \det(x \mathcal{I}\ominus A)$.
The \myemph{eigenvalues} of~$A$ are defined as the corner roots of $f_A$.
\end{df}
Recall that over~$\rmax$ and~$\supertropical$, $\ominus$ means~$\oplus$, and~$f_A$ is  called the maxpolynomial.
Also, over $\rmax$, the eigenvalues of $A$ 
are the roots of $f_A$ in the sense of
tropical geometry, so they coincide with the algebraic tropical eigenvalues
in~\cite{TBE,MPA16}.

The coefficient of~$X^{ k}$ in the formal characteristic polynomial of~$A$ times $(\ominus \unit)^{n-k}$ is the sum of the  
determinants of its~$n-k\times n-k$ principal sub-matrices (that is, obtained by deleting~$k$ chosen rows, and their 
corresponding columns). Thus, this is the trace of the~$(n-k)$th compound matrix of~$A$:
\begin{equation}\label{coef-fA}
f_A= \bigoplus_{k=0}^n (\ominus \unit)^{n-k} \tr(A^{\wedge n-k}) X^k\enspace .
\end{equation}
The combinatorial motivation for the tropical characteristic polynomial is the Best Principal Submatrix problem, and has 
been studied by Butkovic in ~\cite{CCP} and ~\cite{ETCP}.

Recall that orders over $\TT$ are applied to polynomials 
coefficient-wise.
Moreover, polynomials with possibly negative exponents can be composed formally.

\begin{thm}\label{CP} 
Let $A,E\in \TT^{n\times n}$ and $m\in\mathbb{N}$. We have
\begin{subequations}\label{CPeq}
 \begin{align}&f_{E^\nabla A E}\succeq^\circ f_A,&&\text{when $E$ is nonsingular}
\enspace ;\label{succeq1}\\ 
&f_{A^\nabla}\succeq^\circ \det(A)^{ -1} X^{ n} f_A(X^{ -1}),&&\text{when $A$ is nonsingular}\enspace ;\label{succeq2}\\ 
&f_{A^{ m}}\succeq^\circ \bigoplus_{k=0}^n\big(f_A \big)_k^{m}X^k,&&
\text{when}\; \TT=\Mod\; \text{or}\;
\supertropical 
\enspace.\label{models}
\end{align}
\end{subequations}
Moreover,~\eqref{models} implies
\begin{equation} 
f_{A^{ m}}(x^m)\succeq^\circ \big(f_A(x) \big)^{m}\quad \forall x\in\TT
\enspace .\label{models2}\end{equation}
\end{thm}
\begin{proof}
The three first inequalities follow from~\eqref{coef-fA}, together with
Theorem~\ref{CEP}, Corollary~\ref{QEP} and  Corollary~\ref{compowtr},
respectively.
The last one follows from Theorem~\ref{FP}.
\end{proof}
Note that~\eqref{models2} concerns the polynomial functions
 $f_{A^{ m}}$ and $f_A$, and that the inequality is false for the corresponding 
formal polynomials.

\begin{cor}\label{cp1} Assume that $\TT\neq \TT^\circ$.
Equality holds for those coefficients in~\eqref{CPeq}
such that the coefficient in the 
left hand side is in $\TT^{\vee}$. In particular, if $f_M\in\TT^{\vee}[X]$,
then
\begin{align*}
&f_M=f_A,&& \text{for $M=E^\nabla A E$ and $E$ nonsingular}\enspace;\\
&f_M=\det(A)^{ -1} X^n f_A(X^{-1}),&&\text{for}\; M=A^\nabla\;\text{and $A$ nonsingular} \enspace;\\
&f_M= \bigoplus_{k=0}^n\big(f_A \big)_k^{m}X^k,&&\text{for}\; M=A^{ m}\;\text{and}
\; \TT=\supertropical
\enspace .\end{align*}
Moreover, if~$A^\nabla$ is nonsingular and~$f_{A^{\nabla\nabla}}\in\TT^{\vee}[X]$, then~$f_{A^{\nabla\nabla}}=f_{A}$.
\end{cor}
\begin{proof} This is straightforward from Theorem~\ref{CP} and Proposition~\ref{prop2}.
\end{proof}

\begin{cor}\label{CPNS}  Assume that $\TT\neq \TT^\circ$,
$A,E,M\in \TT^{n\times n}$, $m\in\mathbb{N}$ and  $g$ is an invertible map 
satisfying one of the following conditions:
\begin{subequations}\label{cases}
\begin{align}
&M=E^\nabla A E,&&g:\TT\to\TT, x\mapsto x,&&\text{and $E$ is nonsingular}\enspace;\label{M1}\\
&M=A^\nabla, &&g:\TT^*\to\TT^*, x\mapsto x^{-1},&&\text{and $A$ is nonsingular} \enspace;\label{M2}\\
&M=A^{ m}, &&g:\TT\to\TT, x\mapsto x^{m},&&\text{and}\; \TT=\supertropical 
\enspace .\label{M3}\end{align}
\end{subequations}

We have 
\begin{enumerate}
\item\label{cp2}  if $\gamma$ is a root of $f_A$, then $g(\gamma)$
is a root of $f_M$; 
\item\label{cp3} if $\lambda$ is an eigenvalue of $M$,   then
 $g^{-1}(\lambda)$ is an eigenvalue of~$A$.
\end{enumerate}
\end{cor}

The proof of Corollary~\ref{CPNS} uses the following general lemmas.

\begin{lem}\label{lem-roots}
Let $P,Q\in\TT[X]$  with $\TT\neq \TT^\circ$
and assume that $P \succeq^\circ Q$.
Then the degree of $P$ is greater or equal to the degree of $Q$ and 
we have:
\begin{enumerate}
\item\label{root1} 
If $r$ is a root of $Q$, then $r$ is a root of $P$.
\item\label{root2}
If $r$ is a corner root of $P$, then $r$ is a corner root of $Q$.
\end{enumerate}
\end{lem}
\begin{proof} Let $P,Q\in\TT[X]$ be given by
$P=\oplus_{k=0}^n p_k X^k$ and $Q=\oplus_{k=0}^n q_k X^k$
and such that $p_k \succeq^\circ q_k$ for all
$k=0,\ldots, n$, with $p_n$ or $q_n$ possibly equal to zero.
If $p_n=\zero$ (that is if the degree of $P$ is less than $n$) then
$q_n\preceq^\circ \zero$ and since $\zero\in \TT^{\vee}$,
we get that $q_n=\zero$ by Proposition~\ref{prop2},
so the degree of $Q$ is also less than $n$.
This shows that the degree of $P$ is greater or equal to the degree of $Q$.

\eqref{root1}  If $r\in\TT^\vee$ is a root of $Q$, then~$Q(r)\in\TT^{\circ}$.
Since an inequality between formal polynomials implies the same for the
corresponding polynomial functions, we have
$Q(r)\preceq^\circ P(r)$.
This implies that $P(r)\in \TT^{\circ}$ and so $r$ is a
root of $P$.

\eqref{root2}
Let $r\in\TT^\vee$ be a corner root  of $P$.
If $r=\zero$ this means that $P(\zero)=\zero$.
Then $Q(\zero)\preceq^\circ \zero$, which implies that 
$Q(\zero)=\zero$ by Proposition~\ref{prop2}.
So $r$ is a corner root of $Q$.
If $r\in\TT^*$, then by definition $P(r)\in\TT^\circ$ and by
Point~\eqref{corn5} of  Lemma~\ref{lem-corn},
 $P(r)=\oplus_{k \in S} \, p_k r^k$ for some set 
$S\subset\{0\leq k\leq n \mid p_k\in \TT^*\}$.
Let $k\in S$. Since $p_k \succeq^\circ q_k$, 
and  $p_k\in \TT^*\subset \TT^{\vee}$, Proposition~\ref{prop2} implies that
$p_k=q_k$.
So $P(r)=\oplus_{k \in S}\, p_k r^k= \oplus_{k \in S} \,q_k r^k\preceq
Q(r)$ and since we also have $Q(r)\preceq^\circ P(r)$,
so $Q(r)\preceq P(r)$, we deduce that $P(r)=Q(r)= \oplus_{k \in S} \,q_k r^k$.
In particular $Q(r)\in \TT^\circ$, therefore $r$ is a root of $Q$.
Moreover, since  $q_k=p_k\in \TT^*$ for all $k\in S$ and
$Q(r)= \oplus_{k \in S} \, q_k r^k$, we obtain
by Point~\eqref{corn5} of  Lemma~\ref{lem-corn} that
$r$ is a corner root of $Q$.
\end{proof}

\begin{lem}\label{lem-aux}
Assume that $\TT\neq \TT^\circ$, $m\in\mathbb{N}$, $f=\oplus_{k=0}^n f_k X^k$, $Q\in\TT[X]$, and $g$ is an invertible map 
satisfying one of the following conditions:
\begin{subequations}\label{CPeq-pol}
\begin{align}
&Q= (f_0)^{-1} X^{ n} f(X^{ -1}),&& g:\TT^*\to\TT^*, x\mapsto x^{-1},&&
\text{$f_0\in \TT^*$ and $f_n=\unit$} \enspace;
\label{pol2}\\ 
&Q=\bigoplus_{k=0}^n \big(f_k\big)^{m} X^k, && g:\TT\to\TT, x\mapsto x^{m},&& \text{and}\;\TT=\supertropical 
\enspace .
\label{pol3}
\end{align}
\end{subequations}
We have
\begin{enumerate}
\item\label{cp2q}  $\gamma$ is a root of $f$ if, and only if, $g(\gamma)$
is a root of $Q$;
\item\label{cp3q} $\gamma$ is a corner root of $f$ if, and only if, $g(\gamma)$
is a  corner root of $Q$. 
\end{enumerate}
\end{lem}
\begin{proof}
In Case~\eqref{pol2}, $f_0\in \TT^*$, $f_n=\unit$,
and $Q= \oplus_{k=0}^n q_{k} X^{k}$ with 
 $q_{k}=(f_0)^{-1}f_{n-k}$.
If $\gamma\in \TT^\vee$ is a root of $f$, then  $f(\gamma)\in \TT^\circ$.
Since $f(\zero)=f_0\in \TT^*$, and $\TT^*$ and $\TT^\circ$ are
disjoint, $\zero$ is not a root of $f$,
so $\gamma\in \TT^*$ and $g(\gamma)$ exists
and is in $\TT^\vee$.
Then $Q(g(\gamma))= (f_0)^{-1}\gamma^{-n} f(\gamma)\in \TT^\circ$,
so $g(\gamma)$ is a root of $Q$.
If now $\gamma\in \TT^\vee$ is a corner root of $f$,
then $g(\gamma)$ is a root of $Q$.
Moreover, $f(\gamma)=\oplus_{k \in S} f_k \gamma^k$,
where $S\subset\{k\geq 0 \mid 
f_k\in \TT^*\}$. 
This implies that $Q(g(\gamma))=(f_0)^{-1}\gamma^{-n} f(\gamma)
=\oplus_{k \in S} q_{n-k} g(\gamma)^{n-k}$,
with $q_{n-k} \in \TT^*$ 
for all $k\in S$.
Then, by Point~\eqref{corn5} of  Lemma~\ref{lem-corn},
$g(\gamma)$ is a corner root of $Q$.
Using that $f=(q_0)^{-1}X^{ n} Q(X^{ -1})$ and $q_0=(f_0)^{-1}\in \TT^*$ and
$q_n=\unit$, we obtain also the reverse
implications, which shows~\eqref{cp2q} and~\eqref{cp3q} for Case~\eqref{pol2}.

Let us now consider the case~\eqref{pol3} in which $\TT=\supertropical$.
Let us first remark that $g$ is a bijection from $\TT$ 
(resp.\ $\TT^*$,  $\TT^\vee$, $\TT^\circ$) to itself.
If $\gamma\in \TT^\vee$ is a root of $f$, then  $f(\gamma)\in \TT^\circ$.
By Theorem~\ref{FP}, we have 
$Q(g(\gamma))= \oplus_{k=0}^n \big(f_k\big)^{m} \gamma^{km}
= \oplus_{k=0}^n \big(f_k \gamma^{k}\big)^m
=\left(\oplus_{k=0}^n f_k \gamma^{k}\right)^m
=g(f(\gamma))$.
Since $g$ is a bijection from $\TT^\vee$ 
(resp.\ $\TT^\circ$) to itself, we get that $g(\gamma)\in \TT^\vee$
and $Q(g(\gamma))\in \TT^\circ$, so $g(\gamma)$ is a root of $Q$.
Conversely, if $g(\gamma)\in \TT^\vee$ is a root of $Q$, then 
 $g(f(\gamma))\in\TT^\circ$, and since $g$ is a bijection from $\TT^\vee$ 
(resp.\ $\TT^\circ$) to itself,
$\gamma\in \TT^\vee$ and $f(\gamma)\in\TT^\circ$ and so $\gamma$ is a root of $f$.
This shows Point~\eqref{cp2q}. 

Now, let $\gamma$ be a corner root of $f$, then $g(\gamma)$ is a
root of $Q$. Moreover, $f(\gamma)=\oplus_{k \in S} f_k \gamma^k$,
where $S\subset\{k\geq 0 \mid 
f_k\in \TT^*\}$.  
By Theorem~\ref{FP}, we have 
$Q(g(\gamma))
=\left(f(\gamma)\right)^m=\left(\oplus_{k \in S} f_k \gamma^k\right)^m
= \oplus_{k\in S} \big(f_k\big)^m \gamma^{km}
= \oplus_{k\in S} q_k (g(\gamma))^{k}$,
where $q_k=(f_k)^m$ is the $k$th coefficient of $Q$
and $q_{k} \in \TT^*$  
for all $k\in S$. Then, by Point~\eqref{corn5} of  Lemma~\ref{lem-corn},
$g(\gamma)$ is a corner root of $Q$.
Conversely, if $g(\gamma)\in\TT^\vee$ is a corner root of $Q$, then
$\gamma$ is a root of $f$ and
$Q(g(\gamma))= \oplus_{k\in S} q_k (g(\gamma))^{k}$,
where $S\subset\{k\geq 0 \mid  
q_k\in \TT^*\}$. 
This implies that $\left(f(\gamma)\right)^m=Q(g(\gamma))=
 \oplus_{k\in S} q_k (g(\gamma))^{k}
=\left(\oplus_{k \in S} f_k \gamma^k\right)^m$ and since $g$ is a bijection,
we deduce that $f(\gamma)= \oplus_{k \in S} f_k \gamma^k$.
Moreover, $f_{k}=g^{-1}(q_k) \in \TT^*$ 
for all $k\in S$.
Then, by Point~\eqref{corn5} of  Lemma~\ref{lem-corn},
$\gamma$ is a corner root of $f$.
\end{proof}

\begin{proof}[Proof of Corollary~\ref{CPNS}]
Consider the polynomials $P,Q,f\in\TT[X]$ of degree $n$ 
with $P=f_M$, $f=f_A$ and $Q= f_A$  in case~\eqref{M1} and $Q$ as in~\eqref{pol2}
and~\eqref{pol3} in cases~\eqref{M2} and~\eqref{M3} respectively.
By Theorem~\ref{CP}, we have $P \succeq^\circ Q$.
By definition, an eigenvalue of a matrix $M$ is a corner root of $f_M$,
and since $P=f_M$, Lemma~\ref{lem-roots} shows that 
if $\lambda$ is a root of $Q$ then $\lambda$ is a root of $f_M$ 
and if $\lambda$ is an eigenvalue of $M$, 
then $\lambda$ is a corner root of $Q$.
The assertions of the corollary follow in Case~\eqref{M1} since $Q=f_A$
and $g$ is the identity map.
They also follow 
in Cases~\eqref{M2} and~\eqref{M3}, using  Lemma~\ref{lem-aux}.
\end{proof}

\begin{exa}\label{examajo}Let $A=\left(\begin{array}{cc}3&2^\circ\\1&1\end{array}\right)$. Then
 $$A^\nabla=\left(\begin{array}{cc}-3&(-2)^\circ\\\ominus(-3)&-1\end{array}\right)\text{ over }\smax,
\text{ and }A^2=\left(\begin{array}{cc}6&5^\nu\\4&3^\nu\end{array}\right)\text{ over }\supertropical.$$ The 
corresponding characteristic polynomials are
\begin{align*}
f_A=&X^2\ominus 3X\oplus 4\quad\text{over}\; \smax\enspace ,\; \text{and}\;
f_A=X^2\oplus 3X\oplus 4\quad\text{over}\;\supertropical\enspace , \\ 
f_{A^\nabla}=&X^2\ominus (-1)X\oplus (-4)=(-4)X^2(4\ominus 3X^{-1}\oplus X^{-2})
= \det(A)^{ -1} X^{ n} f_A(X^{ -1})\enspace,\\
f_{A^2}=&X^2\oplus 6X\oplus 9^\nu
\models X^2\oplus 6X\oplus 8
= X^2\oplus (f_A)_1^2 X \oplus (f_A)_0^2\enspace .
\end{align*}
The polynomials $f_A$ and $f_{A^\nabla}$ have only two roots:
$3$  and $1$ for $f_A$ (either in $\supertropical$ or $\smax$)
and $(-3)=3^{-1}$  and $(-1)=1^{-1}$
for $f_{A^\nabla}$. These roots are also corner roots
so eigenvalues of $A$ and $A^\nabla$ respectively.
The polynomial $f_{A^2}$ has a unique corner root $3^2$,
that is $A^2$ has a unique eigenvalue, whereas 
all the $x\in \supertropical^\vee$ such that $x\leq 1.5^2$ are roots of $A^2$.
\end{exa}

In Corollary~\ref{CPNS}, we related the eigenvalues of the matrices of
Theorem~\ref{CP}, under the assumption that $\TT\neq \TT^\circ$.
The typical example where $\TT=\TT^\circ$ is when $\TT=\Mod=\rmax$.
In that case, roots are any elements, so Point~\eqref{cp2} of 
Corollary~\ref{CPNS} is true but has no interest.
Moreover, $\succeq^\circ$ is simply the order $\succeq$, so
one cannot expect an exact correspondence between eigenvalues
of $A$ and $M$ as in Corollary~\ref{CPNS}.
Nethertheless, one can apply~\cite[Lemma~4.2]{MPA16} to
 obtain the following majorization inequality.
Recall that a matrix $A$ over $\rmax$ is nonsingular if and only if
 $\per A\neq \zero$.
For a polynomial $P$ over $\rmax$, we define as in~\cite{MPA16}, 
the multiplicity of a corner root $r$ as the difference between right and
left slopes of the polynomial function $P$ at point $r$.

\begin{cor}\label{CPNS2}  Assume that $\TT=\rmax$,
$A,E,M\in \TT^{n\times n}$, $m\in\mathbb{N}$ and  $g$ is an invertible map 
satisfying one of the following conditions:
\begin{subequations}\label{cases2}
\begin{align}
&M=E^\nabla A E,&&g:\TT\to\TT, x\mapsto x,&&\text{and $E$ is nonsingular}\enspace;\label{M10}\\
&M=A^\nabla, &&g:\TT^*\to\TT^*, x\mapsto x^{-1},&&\text{and $A$ is nonsingular} \enspace;\label{M20}\\
&M=A^{ m}, &&g:\TT\to\TT, x\mapsto x^{m}.&& \label{M30}\end{align}
\end{subequations}
Let $\lambda_1,\lambda_2,\ldots ,\lambda_n$ and~$\gamma_1,\gamma_2,\ldots,\gamma_n$ denote 
the eigenvalues of~$M$ and~$A$, respectively, counted with multiplicities
and ordered so that $\lambda_1\succeq\lambda_2\succeq\dots\succeq\lambda_n$ 
and $g(\gamma_1)\succeq g(\gamma_2)\succeq\dots\succeq g(\gamma_n)$.
Then
\[ \lambda_1\cdots \lambda_k\succeq g(\gamma_1)\cdots g(\gamma_k)\quad \forall k\in[n].
\]
Moreover, in Case~\eqref{M20}, the latter inequality becomes an equality 
for $k=n$.
\end{cor}
\begin{proof}
\cite[Lemma~4.2]{MPA16} states that if $P=\oplus_{k=0}^n p_k X^k,\;
Q=\oplus_{k=0}^n q_k X^k\;\in\rmax[X]$ are such that
$P\succeq Q$ with $p_n=q_n$,
and the corner roots of~$P$ and~$Q$ are respectively
$\lambda_1\succeq\lambda_2\succeq\dots \succeq\lambda_n$ and~$\delta_1\succeq\delta_2\succeq\dots\succeq\delta_n$, counted with multiplicities, then
\begin{equation}
\label{major}
\lambda_1\cdots \lambda_k\succeq \delta_1\cdots \delta_k\quad \forall k\in[n]\enspace .
\end{equation}
If in addition $p_0=q_0$, then~\eqref{major} becomes an equality 
for $k=n$.

Consider, as in the proof of Corollary~\ref{CPNS},
the polynomials $P,Q,f\in\rmax[X]$ of degree $n$ 
with $P=f_M$, $f=f_A$ and $Q= f_A$  in case~\eqref{M10} and $Q$ as in~\eqref{pol2}
and~\eqref{pol3} in cases~\eqref{M20} and~\eqref{M30} respectively.
By Theorem~\ref{CP}, we have $P \succeq Q$.
Moreover $p_n=q_n=\unit$ in all cases.
So~\eqref{major} holds in all cases for the corner roots 
$\lambda_1\succeq\dots \succeq\lambda_n$ of $P$ 
and~$\delta_1\succeq\dots\succeq\delta_n$ of $Q$.
The corner roots of $P$ are the eigenvalues of $M$ by definition.
Moreover, the corner roots of $Q$ are the eigenvalues of $A$ 
in Case~\eqref{M10} since $Q=f_A$. It is easy to see from the definition of 
corner roots and multiplicities,  that Point~\eqref{cp3q}
of Lemma~\ref{lem-aux} holds true for $\TT=\rmax$,
that is the corner roots of $Q$ are the images by $g$ of the
corner roots of $f$ and that in addition the multiplicities coincide.
So in Cases~\eqref{M2} and~\eqref{M3},
the corner roots of $Q$ are the images by $g$ of the
eigenvalues of $A$, so $\delta_i=g(\gamma_i)$.
This shows the first assertion of the corollary.

Now, in Case~\eqref{M2}, we have $p_0=\det(A^\nabla)$ and
$q_0=\det(A)^{-1}$. By Point~\eqref{cor342} of Proposition~\ref{defbar}
applied to $\TT=\rmax$, we obtain
that $p_0=q_0$, hence~\eqref{major} becomes an equality 
for $k=n$, which shows the last assertion of the corollary.
\end{proof}

\bibliographystyle{alpha}

\bibliography{tropical}

\end{document}